\pdfoutput=1
\documentclass[preprint, authoryear]{elsarticle}

\usepackage[export]{adjustbox}
\usepackage{graphics}\graphicspath {{./resources/}}
\usepackage{amsmath}
\usepackage{amsthm}
\usepackage{algorithm2e}
\usepackage{mathtools}
\usepackage{booktabs}
\usepackage{pgfplots}
\usepackage{bm}

\begin{document}

\title{Product-Closing Approximation for \textcolor{blue}{Ranking-based} Choice Network Revenue Management}

\author[poly]{Thibault Barbier \corref{cor1}}
\ead{thibault.barbier@polymtl.ca}

\author[poly]{Miguel Anjos}
\ead{miguel-f.anjos@polymtl.ca}

\author[expr]{Fabien Cirinei} 
\ead{fabien.cirinei@expretio.com}

\author[poly]{Gilles Savard}
\ead{gilles.savard@polymtl.ca}

\cortext[cor1]{Corresponding author}

\address[poly]{Department of Industrial and Mathematical Engineering , Ecole Polytechnique Montreal, 2900 Edouard Montpetit Blvd, Montreal, QC H3T 1J4 Canada}
\address[expr]{ExPretio Technologies, 200 Avenue Laurier Ouest, Montreal, QC H2T 2N8, Canada}

\begin{abstract}

Most recent research in network revenue management incorporates choice behavior that models the customers' buying logic. 
These models are consequently more complex to solve, but they return a more robust policy that usually generates better expected revenue than an independent-demand model. 
Choice network revenue management has an exact dynamic programming formulation that rapidly becomes intractable.
Approximations have been developed, and many of them are based on the multinomial logit demand model. 
However, this parametric model has the property known as the independence of irrelevant alternatives and is often replaced in practice by a nonparametric model.
We propose a new approximation called the product closing program that is specifically designed for \textcolor{blue}{a ranking-based choice model representing a nonparametric demand}.
Numerical experiments show that our approach quickly returns expected revenues that are slightly better than those of other approximations, especially for large instances.
Our approximation can also supply a good initial solution for other approaches.
\end{abstract}


\maketitle
\newtheorem{proposition}{Proposition}
\newtheorem{lemma}{Lemma}

\section{Introduction}
\label{productClosing:sec:introduction}

In 1978, when the US airline market was deregulated, airlines lost their quasi-monopolistic status, moving to a competitive market. 
They were forced to improve efficiency, in terms of both operation productivity and sales profitability.
Operation productivity optimization aims to improve the scheduling, maintenance, and assignment of limited resources.
Sales profitability optimization is a type of revenue management: it aims to maximize the revenue obtained from perishable resources. 
These issues are considered separately because the subproblems are tractable whereas the overall problem is too complex.
Today, scheduling and revenue management have many applications: airlines, rental car companies, and hotels. 

We focus on the revenue management problem for which perishable resources are sold through different products to customers during a reservation period.
Selling a low price product early consumes a resource that could perhaps have been sold later at a better price. However, holding on to resources for future sales fails to satisfy the current demand.

The challenge is thus to control the availability of products over the horizon period to maximize revenue.
A policy designates all the logics to control the availability of products throughout the entire reservation period.
The resources and products are parameters fixed by the user.
The forecast of the demand is not the subject of this article and is considered given.
The revenue management in this article refers to the problem of availability policy. 

Research has shown that it is better to optimize the network formed by the resources rather than each resource individually, but this leads to larger problems.
The latest trend in revenue management is the implementation of choice behavior instead of an assumption of independent demand.
The problem is more complex, but the solutions are more accurate and robust.
This version of revenue management is referred to as the choice network revenue management problem (CNRM).  
It was first introduced by \cite{gallego2004managing}, and an exact dynamic programming (DP) formulation was given by \cite{talluri2004revenue}.

However, the DP rapidly becomes intractable because of the number of states.
Researchers have therefore proposed various approximations, returning solutions that are either dynamic or static.
The quality of the approximation can be measured by the expected revenue and the solution time.
The most popular approximations are the choice deterministic linear program (CDLP) proposed by \cite{liu2008choice}, which is static, and DP decomposition by resources, which is dynamic.
For large instances and especially because of the choice behavior, these approximations are large and difficult to solve.
The multinomial logit (MNL) model as explained in \cite{ben1985discrete} and \cite{hanson1996optimizing}, which is widely used in the marketing and economics literature, is often used for the choice behavior.
Many methods such as column generation and heuristics have been developed for this model because its structure is well-accommodated for estimation and CNRM approximations.s

However, the MNL model has an important drawback known as the independence of irrelevant alternatives (IIA) as detailed in \cite{ratliff2008multi}.
IIA causes improbable substitutions when products share similar characteristics.
Unfortunately perfect substitutes, such as the red/blue bus example of \cite{ben1985discrete}, often occur in revenue management. 
Moreover, the data available for forecasting may better fit another demand model.
We focus on the preference list (PL) model which is \textcolor{blue}{one} nonparametric alternative to the multinomial logit.
\textcolor{blue}{It is also referred as the ranking-based model}.
In the former, customers choose from an ordered list of ranked products.
A probability of transition is specified between each pair of products.
Our work is motivated by the fact that most recent researches on choice behavior models have focused on PL estimation as in \cite{farias2013nonparametric}, \cite{van2014market} and more recently \cite{van2017expectation}.
On the other hand, there has been limited research into \textcolor{blue}{ranking-based} CNRM approximations, and most of the studies are adaptations of existing MNL approaches. 

Our approximation exploits the structure of the PL model and is not based on an existing approximation. 
By taking advantage of the logical transitions between products rather than working with sets of products as in MNL approximations, we avoid the extremely high number of product combinations.
This results in a nonlinear model that can easily be linearized, and the binary variables have a practical significance that can be exploited to provide good initial solutions.
The complexity of our model depends linearly on the number of products considered for each segment.
Unlike many other approximations, our formulation benefits from overlapping by reusing variables when different segments share products; this reduces the complexity.
We assume no-reopening: products are sold until a specified time and then never sold again.
Some companies have such a strategy, and most approximations model it via additional constraints that slow the solution process.
When reopening is allowed, our approximation can return a set of product durations that can serve as a good initial solution for an approximation that allows reopening.

Our experiments show promising results in comparison with other approximations.
Our approximation returns an equivalent or better expected revenue in a shorter solution time for all the instances, although there is no-reopening.
The results also demonstrate the time saved by using our solution as an initial solution for an approximation with reopening.
We also show the limitations of some current approximations for the largest instances, to highlight the practical feasibility of our approach.

The remainder of this paper is organized as follows.
In Section~\ref{productClosing:sec:relatedLiterature}, we review the CNRM literature, especially with \textcolor{blue}{ranking-based} choice behavior.
In Section~\ref{productClosing:sec:notationsAndPreviousApproaches}, we introduce the notation and give the exact formulation of CNRM.
In Section~\ref{productClosing:sec:closingApproximation}, we present our approximation with preference-list choice behavior and its theoretical properties.
In Section~\ref{productClosing:sec:solvingThePCP}, we present practical methods for the efficient solution of our approximation.  
Numerical experiments and approximation benchmarks are reported in Section~\ref{productClosing:sec:numericalExperiments}, and
Section~\ref{productClosing:sec:conclusion} provides concluding remarks. 

\section{Related literature} 
\label{productClosing:sec:relatedLiterature}

We refer to \cite{talluri2006theory} for reviews of the historical revenue management problem with or without the network and choice aspects.
\cite{strauss2018review} presents the most recent researchs on the general revenue management with choice behavior.
We focus on the CNRM problem and discuss only the relevant literature.

As mentioned in the Introduction, this problem has an exact DP formulation \cite{talluri2004revenue}.
Because it rapidly becomes intractable, approximations have been proposed in two categories: static and dynamic.

The static approximations are based on the expected demand.
They therefore reduce the complexity and can tackle larger instances but ignore the demand uncertainty.
The solution obtained is not updated in response to new arrivals and is hence called static.  
In this category, CDLP \cite{liu2008choice} is the most widely used.
It indicates for how long each set of products, also called an offer, must be sold over the reservation period.
By empirically ordering the offers and their durations over the reservation period, we obtain a static policy by offer period.
The CDLP is an upper bound on the DP solution and is asymptotically equivalent as resources and demand increase.
However, it has an exponential number of columns and must be solved by column generation, which has an NP-hard subproblem, as explained by \cite{bront2009column} and \cite{rusmevichientong2014assortment}.
\cite{liu2008choice} and \cite{bront2009column} propose exact and heuristic subproblem formulations for the MNL choice behavior.

The CDLP primal solution has to be ordered and gives a static policy.
\cite{liu2008choice} and \cite{bront2009column} use the optimal dual values to calculate the capacity marginal values in a DP decomposition by resource.
\cite{zhang2016dynamic} and \cite{erdelyi2011using} are other approximations for the calculus of the network marginal values. 
In the same vein \cite{kunnumkal2010new} uses the revenue attributed to each resource rather than dual values.
The dynamic policy obtained indicates what offer to propose as a function of the remaining time and capacities.
However, this approach needs to solve an NP-hard problem for each resource, each time, and each remaining capacity and can therefore be intractable even if computed offline.
Moreover, an NP-hard problem must be solved for each arrival to determine what offer to propose.
This may be incompatible with current reservation systems.

\cite{talluri2010randomized} proposed the segment-based deterministic concave program (SDCP), considered as a CDLP decomposition by segment. 
It is more tractable if the consideration sets are not too large, but it also provides a weaker upper bound than CDLP unless the segments do not overlap, which is rare in practice. 
To tighten the SDCP formulation with choice behavior, \cite{meissner2013enhanced} add valid constraints referred to as product cuts, \cite{talluri2014new} uses a random customer-arrival stream and \cite{strauss2017tractable} proves an equivalence with CDLP when the intersection of segment consideration forms a tree or consideration sets are nested.
However, even with the extra constraints, no primal policy is returned and the dynamic decomposition is the principal solution. 
The sales-based linear program (SBLP) introduced by \cite{gallego2014general} and developed further by \cite{talluri2014new} is a compact formulation of the SDCP under the MNL choice behavior.
	
\textcolor{blue}
{
	Previous researches focus on parametric choice model.
	\cite{chaneton2011computing} was one of the first article on the CNRM policy problem with a ranking-based choice behavior. 
	They formulated a continuous demand and capacity model solved with a stochastic gradient descent.
	Stochastic gradient descent performance largely depends on the initial solution and the stop criterion.
	Therefore, this approach does not ensure a good solution for any instance.
	Later, \cite{hosseinalifam2016computing} developed a subproblem for the CDLP column generation problem that tackle a ranking-based choice model. 
	However this subproblem is an integer program difficult to solve for larger instances.
	More recently, \cite{paul2018assortment} solve a tree choice model for the assortment problem.
	They propose a polynomial time algorithm to solve a dynamic program under some assumptions.
	The recent advances on \textcolor{blue}{ranking-based} choice behavior forecast, such as \cite{farias2013nonparametric}, \cite{van2014market},  \cite{jagabathula2017nonparametric} and \cite{van2017expectation}, open the way to new researches and approaches for the CNRM policy problem under this choice behavior. 
}

\section{Notations and previous approaches} 
\label{productClosing:sec:notationsAndPreviousApproaches}

\subsection{Notations}
\label{productClosing:subsec:notations} 

We start by introducing the notation for the CNRM problem.
A resource $i \in I$ has a capacity $c_i$. There is $m=|I|$ resources.
A product $j \in J$ is defined by a fare $r_j$ and consumed one or more resources. There is $n=|J|$ products.
An offer $S \subseteq J$ denotes a set of distinct products.
The incidence matrix $A=[a_{ij}]_{i \in I, j \in J}$ has $a_{ij}$ equal to $1$ if resource $i$ is used by product $j$ and $0$ otherwise.
$A_j$ refers to the column of product $j$ in the incidence matrix.
Customers arrive during a reservation period, indexed by $t$, starting at $t=0$ and finishing at $t=\tau$ when the resources perish.
A segment $l \in L$ groups the customers with identical choice behavior aiming to buy products $C_l \subseteq J$, also called the consideration set and containing $n_l$ products.
These customers arrive over the reservation period according to a Poisson process with a uniform ratio $\lambda_l$. 
The choice behavior is defined by the probability $\mathsf{P}_l(j|S)$ that segment $l$ buys product $j$ among the offer $S\subseteq J$.
We focus on preference-list \textcolor{blue}{also called ranking-based} choice behavior, which is nonparametric. 
It is characterized by distinct ordered products indexed by $l_j \in [1,n_l]$ or $l_j =0$ if $j \not \in C_l$.
The product $l^k \in C_l$ is the $k^{th}$ product of the preference list of segment $l \in L$.
The subset $C_l^{[k]} \subseteq C_l$ with $k \in [1, n_l]$ corresponds to the preference list limited to the first $k$ products.
A transition $\theta^{k\text{-}1,k}_l$ with $k \in [1, n_l]$ reflects the ratio of customers passing from one product to the next in the preference list.
By convention, we fix $\theta^{0,1}_l = 1$.
Customers always choose a product according to the order defined by the preference list.
We therefore have:

\begin{equation*}
\mathsf{P}_l(j|S) = 
\left\{
\begin{array}{ll}
\prod\limits_{k=1}^{l_j} \theta_l^{k\text{-}1,k}	& \text{if } S \cap C_l^{[l_j]} = \{j\} \\
0								& \text{otherwise.}
\end{array}
\right. 
, \quad
\forall l \in L, \; j \in J, \; S\subseteq J.
\end{equation*}

We often shorten the preference-list notation to $l^1 \xrightarrow{\theta^{1,2}_l} l^2 \xrightarrow{\theta^{2,3}_l} \dots \xrightarrow{\theta^{n_l\text{-}1,n_l}_l} l^{n_l}$.

To be noted that the transitions allow us to merge preference lists without transitions.
For example $2: u \rightarrow v$ and $3: u \rightarrow v \rightarrow z$ can be merged as $5: u \xrightarrow{1} v \xrightarrow{3/5} z$.	

\emph{Running example:} Let consider three products $u$, $v$, $w$ of respective price $15$, $25$ and $40$.
$u$ and $v$ consume the same leg of capacity $c_1 =1$ and $w$ another leg of capacity $c_2=1$. The reservation period horizon is $\tau=1$.
Only one segment $l$ arrives at a ratio $\lambda_l=3$ and has preference list $u \xrightarrow{0.9} v \xrightarrow{0.8} w$.

To illustrate the preference-list choice behavior, imagine we offer $S=\{v,w\}$ to the \emph{running example} segment.
Then the arrival ratio $\mathsf{P}_l(u|S) = 0$ because $u$ is not offered ($u \not \in S$).
For the second choice, $\mathsf{P}_l(v|S) = 0.9$ because $v$ is the first offered choice ($S \cap C_l^{[2]} = \{v\}$).
Finally $\mathsf{P}_l(w|S) = 0$ because a preferred choice is offered ($S \cap C_l^{[3]} = \{v,w\} \ne \{w\}$).

\subsection{Dynamic programming formulation} 
\label{productClosing:subsec:dynamicProgramFormulation} 

The CNRM problem can be formulated exactly as a DP.
We choose a step size $\delta t$ sufficiently small that there is at most one arrival between $t$ and $t+\delta t$.
We also introduce $x$, the vector of remaining capacities ($x=c$ when $t=0$).
The Bellman equations can then be written as follows:

\begin{align*}
\label{productClosing:eq:DP}			
\tag{DP}
\mathsf{V}(t\text{+}\delta t,x) =&\; \mathsf{V}(t, x) \\
		&+ \max_{S \subseteq J(x)} \sum_{l \in L} \lambda_l \sum_{j\in C_l}\mathsf{P}_l\left(j|S\right)\left(r_j -\Delta \mathsf{V}_j(t\text{+}\delta t, x)\right)\delta t 
\end{align*}

where $\Delta \mathsf{V}_j(t,x) = \mathsf{V}(t,x) - \mathsf{V}(t, x\text{-}A_j)$ is the opportunity cost of selling product $j$ at time $t$.
$J(x)$ is the set of products with remaining resource capacity.
The boundary conditions are:  
\begin{align*}
\mathsf{V}(t,0) = 0, \quad  		&\forall t \in [0,\tau],\\
\mathsf{V}(\tau+\delta t, x) = 0, \quad  	& 0\le x \le c.
\end{align*}

\textcolor{blue}
{
The optimal policy, denoted by $S^\star(t,x)$, for deciding the availability of each product over the reservation period is formed by the maximization problems solution of $S$ at each time and for each remaining capacity in the previous Equation \ref{productClosing:eq:DP}.
We obtain it by comparing each product price to the opportunity cost of reducing resource required for selling the product:
\begin{equation*}
	S^\star(t,x) = \left\{j \in J \; | \; r_j \ge \mathsf{V}(t+\delta t, x) - \mathsf{V}(t+\delta t, x-A_j)\right\}
\end{equation*}
}
\textcolor{blue}{This policy revenue is the optimal revenue and is denoted by $ \mathsf{V^{\star}}$}.
Unfortunately, this DP rapidly becomes intractable as the size of the network increases.
Even an instance with only ten resources of capacity 100 has $100^{10}$ states.
The CNRM problem must therefore be solved approximately.

\subsection{Static approximations}
\label{productClosing:subsec:staticApproximations}

We first consider static approximations.
They avoid the discrete customer-arrival complexity of the DP by considering a continuous and deterministic flow of customers.
The most popular static approximation is the CDLP \citep{gallego2004managing} based on durations $\mathsf{D} =\{\mathsf{D}_S\}_{S\subseteq J}$ indicating for how much time each offer should be available.
The vector of products sales is denoted $\mathsf{Q}$. For the CDLP, we have for each product:
\begin{equation*}
\mathsf{Q}_{\text{\tiny CDLP}, j} (\mathsf{D})
= \sum_{S\subseteq J}\sum_{l \in L} \lambda_l \mathsf{P}_l(j|S) \mathsf{D}_S
\end{equation*}
Such that:

\begin{align*}
\label{productClosing:eq:CDLP}
\tag{CDLP}
\mathsf{R}^\star_\text{\tiny CDLP} (\mathsf{D^\star}) = \max_\mathsf{D} \quad	
	& r^\top  \mathsf{Q}_{\text{\tiny CDLP}} (\mathsf{D})\\ 
s.t. \quad					
	& A \mathsf{Q}_{\text{\tiny CDLP}}  (\mathsf{D})  \le c, \quad (\pi) \\ 
	& \sum_{S\subseteq J} \mathsf{D}_S \le \tau,															\\
	& \mathsf{D}_S \ge 0, 	\quad \forall S \subseteq J.
\end{align*}

The objective function maximizes the revenue, the first constraint ensure that the capacities are respected and the second constraint prevents from selling more than the reservation period.

The CDLP optimal solution for the \emph{running example} is $\mathsf{D}^\star_{\{v,w\}}=0.370$ and $\mathsf{D}^\star_{\{w\}}=0.463$.
Note that the CDLP suffers from degeneracy especially for preference-list choice behavior. 
In fact, $\mathsf{D}^\star_{\{v\}} = 0.370$ and $\mathsf{D}^\star_{\{w\}}=0.463$ is another optimal solution.	

We must now determine a policy to control the availability of products throughout the reservation period
An immediate policy for any static approximation is the product booking (PB) policy that limit the sales $\mathsf{Q}_j$ for each product to its optimal sales $\mathsf{Q}_j^\star$:

\begin{equation}
\label{productClosing:eq:policyPB}
\tag{PB policy}
S^\text{PB}(t,x) = \{j \in J \; | \; \mathsf{Q}_j \le \mathsf{Q}_j^\star \}, \quad \forall t \in [0,\tau], \; x \le c.
\end{equation}

This policy is therefore static because it is fixed for the entire reservation period.

For the \emph{running example}, the optimal sales are $\mathsf{Q}^\star_v=\mathsf{Q}^\star_w=1$ and $\mathsf{Q}^\star_u=0$.
The PB policy ensures to accept at most one booking for $w$ and $v$.

Practitioners derive the offer period (OP) policy by ordering the $\eta$ offers with non-null duration over the reservation period, such that:

\begin{equation*}
\label{productClosing:eq:policyOP}
\tag{OP policy}
S^\text{OP}(t,x) = \left\{j \in J\; | \; k \in [0, \eta], \; j \in S_{k}, \; t \in [t_k, t_{k+1}[ \right\}, \quad \forall t \in [0,\tau], \; x \le c.
\end{equation*}
Where $k$ define how offers with non-null duration are offered and $t_k = \sum_{\gamma=0}^{k} \mathsf{D}_{S_\gamma}$ the time $S_k$ starts to be offered.

The different orders are equivalent in theory.
As the PB policy, it does not change over the reservation period and is thus static.

For the \emph{running example} we can offer $\{w\}$ then $\{v,w\}$ or inversely.
For the first option, $\{v,w\}$ will be offered from $0$ to $0.463$, then $\{w\}$ from $0.463$ ro $0.463+0.370=0.833$.	

To be noted that we can re-optimized any static approximation several times over the reservation period to obtain a more ``dynamic" PB or OP policy.

\subsection{Dynamic approximations}
\label{productClosing:subsec:dynamicApproximations}

The second type of approximations estimates the pseudo-revenue $r_j - \Delta \mathsf{V}_j(t+h,x)$ of each product without solving the entire DP.
Most of these approaches implement a decomposition by resource to reduce the number of states.
For example, \cite{bront2009column} approximate the network value function for resource $i$ as:
\begin{equation*}
\tag{DCOMP}
\mathsf{V}(t, x) \approx \mathsf{V}_i(t,x_i) + \sum\limits_{k \not= i} \pi^\star_{k} x_k
\label{productClosing:eq:DCOMP}
\end{equation*}
where the dual prices $\pi^\star_k$ come from the optimal solution of a static approximation.
By substituting this expression into the DP we obtain one DP per resource for the calculation of  $\mathsf{V}_i(t,x)$.
The network opportunity cost $\Delta\mathsf{V}_j(t,x)$ can then be approximated by $\Delta \widetilde{\mathsf{V}}_j(t,x)$ based on the decompositions by resource, for example \citep{bront2009column}:
\begin{equation*}
\Delta \mathsf{V}_j(t,x) \approx \Delta \widetilde{\mathsf{V}}_j(t,x) = \sum_{\substack{i \in I \\ a_{ij} = 1}} \beta \Delta \mathsf{V}_i(t,x_i) + (1-\beta) \pi_i^\star.
\end{equation*}
Where $\Delta \mathsf{V}_i(t,x)= \mathsf{V}_i(t,x_i) - \mathsf{V}_i(t,x_i\text{-}1)$ and $0 \le \beta \le 1$ is a parameter to fine-tune.
Other approximations have been proposed by \cite{zhang2016dynamic} and \cite{erdelyi2011using}.
Similarly to the \ref{productClosing:eq:DP}, the policy for the products availability over the reservation period is called the offer dynamic (OD) and is obtained as follows:

\begin{align*}
\label{productClosing:eq:policyOD}
\tag{OD policy}
&S^\text{OD} (t,x)= \\ & \arg\max_{S \subseteq J(x)} \sum_{l \in L} \lambda_l \sum_{j\in S}\mathsf{P}_l\left(j|S\right) \left(r_j -\Delta \widetilde{\mathsf{V}}_j(t\text{+}\delta t,x) \right)  \delta t,  \quad \forall t \in [0,\tau], \; x \le c.  
\end{align*}

This approach is dynamic because it changes depending on the arrivals.

\section{Closing approximation} 
\label{productClosing:sec:closingApproximation}

We propose a new static approximation for the CNRM problem under non-parametric choice behavior.
Our approximation is based on a new policy, which we call product closing (PC), that is suitable for use with a  preference list. 
It determines the time $\mathsf{T}_j \in [0,\tau]$ when each product becomes unavailable such that the policy is:
\begin{equation*}
\label{productClosing:eq:policyPC}
\tag{PC policy}
S^\text{PC}(t,x) = \{j \in J \; | \; t \le \mathsf{T}_j \}, \quad \forall t \in [0, \tau], \; x \le c. 
\end{equation*}	
In other words, it closes the sale of the product at this time. 

The fact that products can not be sold again after being removed from sales is called no-reopening and is defined as follows:
\begin{equation*}
\tag{No-reopening}
\text{No-reopening} \Leftrightarrow \forall j \in  S(t,x), \; t \in [0,\tau] \; \vert \; \forall \xi \in [0,t], \; j \in S(\xi,x)
\end{equation*}
A no-reopening policy is sometimes mandatory in practice.
The PC policy naturally prohibits reopening, whereas OP and OD do not if no constraints are added.
For the \emph{running example}, the CDLP optimal solution reopens sales of product $v$ even if not offered before.

\subsection{Buying logic under closing policy} 
\label{productClosing:subsec:buyinglogicunderclosingpolicy}

To determine the product sales generated by a PC policy, we start by calculating for how long each segment buys each of its choices.
We first note that the $k^{th}$ choice in a preference list is bought provided the corresponding product is available, and the products of the previous choices are not available.
To explain the buying logic driven by the PC policy, we consider the \emph{running example} segment and do not consider the legs capacity issue for the moment.
In Figure \ref{productClosing:fig:buyingLogicExamples} we illustrate two cases (a) and (b) of buying logic depending on the PC times for the segment.
\begin{figure}[h!]
\centering
\includegraphics[max width=\textwidth]{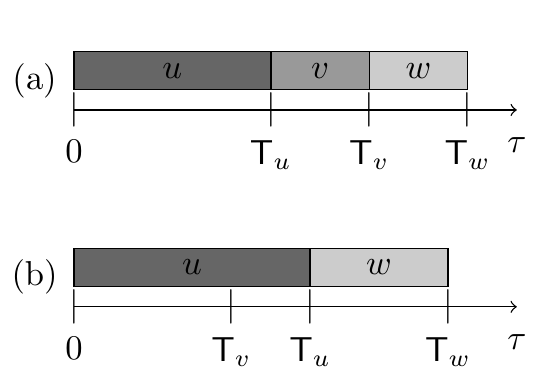}
\caption{Buying logic examples for a segment with preference list $u \xrightarrow{.9} v \xrightarrow{.8} w$. \label{productClosing:fig:buyingLogicExamples}}
\end{figure}
In case (a), the order is $\mathsf{T}_u \le \mathsf{T}_v \le \mathsf{T}_w$, i.e., the segment buys $u$ during $\mathsf{T}_u$, then $v$ during $\mathsf{T}_v-\mathsf{T}_u$, and finally $w$ during $\mathsf{T}_w - \mathsf{T}_v$.
In case (b), the order is $\mathsf{T}_v \le \mathsf{T}_u \le \mathsf{T}_w$, i.e., the segment buys $u$ during $t_u$ and then $w$ during $\mathsf{T}_w - \mathsf{T}_u$ because choice $v$ is covered by choice $u$ as a consequence of $\mathsf{T}_v \le \mathsf{T}_u$.

To generalize the buying logic, we note that the $k^{th}$ choice of a segment is bought if and only if its PC $\mathsf{T}_{l^k}$ is greater than the PCs $\mathsf{T}_{l^\gamma}$ of the previous choices $\gamma\in[1,k[$.
If this condition is satisfied, the choice is bought during the maximum closing $\max\limits_{\gamma \in[1,k\text{-}1]} \; \mathsf{T}_{l^\gamma}$ of the previous choices and its PC $\mathsf{T}_{l^k}$.
We can therefore determine the sales duration $\mathsf{D}_l^k$ for each segment $l$ and choice $k$ as follows:
\begin{equation}
\label{productClosing:eq:salesDuration1}
\mathsf{D}_l^k = \left(\mathsf{T}_{l^k} - \max_{\gamma \in[1,k\text{-}1]} \; \mathsf{T}_{l^\gamma} \right)^+ , \quad  \forall k\in [1,n_l], \; l \in L.
\end{equation}
If we apply this formula to the above example, we find the same durations.
Let $\mathsf{T}_l^k = \max\limits_{\gamma \in[1,k]} \; \mathsf{T}_{l^\gamma} = \max\limits_{j \in C_l^k} \; \mathsf{T}_j$ denote the maximum closing time of the first $k$ products.
With this notation we can reformulate (\ref{productClosing:eq:salesDuration1}) equivalently as:  
\begin{equation}
\label{productClosing:eq:salesDuration2}
\mathsf{D}_l^k = \mathsf{T}_l^k - \mathsf{T}_l^{k\text{-}1}, \quad \forall k\in [1, n_l], \; l \in L.
\end{equation}
We denote by $\mathsf{D}_L = \{\mathsf{D}_l^k\}_{l\in L, k \in[1,n_l]}$ and $\mathsf{T}_L = \{\mathsf{T}_l^k\}_{l\in L, k \in[1,n_l]}$ respectively the segments choice durations and closing times. 

\subsection{Product closing program}
\label{productClosing:subsec:productClosingProgram} 

The quantity that the segment buys is then obtained by multiplying the duration defined in (\ref{productClosing:eq:salesDuration2}) by the buying probability as defined in Section~\ref{productClosing:subsec:notations}, such that:
\begin{equation*}
\mathsf{Q}_{\text{\tiny PCP}}(\mathsf{D_L})_j = \sum\limits_{l \in L}\lambda_l \sum\limits_{j \in C_l} \mathsf{P}_l(j|\{j\}) \mathsf{D}_l^{l_j}
\end{equation*}
We can then write the PC program (PCP) as the following static approximation: 
\begin{align*}
\label{productClosing:eq:PCP}
\tag{PCP}
\mathsf{R}^\star_\text{\tiny PCP} (\mathsf{T^\star}) = \max_\mathsf{T}  \quad
	& r^\top \mathsf{Q}_{\text{\tiny PCP}} (\mathsf{D}_L)	\\
s.t. \quad	
	&A \mathsf{Q}_{\text{\tiny PCP}} (\mathsf{D}_L) \le c,	&	\\
&\mathsf{D}_l^k  = \mathsf{T}^k_l - \mathsf{T}^{k\text{-}1}_l,		& \forall k \in [1,n_l], \; l \in L	,		\\
	&\mathsf{T}^k_l = \max\limits_{j \in C^{[k]}_l} \mathsf{T}_j			& \forall k \in [1,n_l], \; l \in L			\\
	&\mathsf{T}_j \in [0,\tau]	 										& \forall j \in J
\end{align*}

The objective function maximizes the revenue, the first constraint ensure that the capacities are respected, the second determines sales duration with segment choice closings and the third links these latter to product closings.	

For the \emph{running example}, the optimal product closing policy is $\mathsf{T}^\star_u =0$, $\mathsf{T}^\star_v =0.370$ and $\mathsf{T}^\star_w =0.832$.	
It is thus less subject to degeneracy than CDLP because it depends on products rather than offers.

Contrarily to the CDLP introduced in section \ref{productClosing:subsec:staticApproximations}, the PCP do not directly returns the capacity dual variables ($\pi$) because it is a mixed integer program.
To use dynamic approximations, we approximate them by solving the PCP $m$ times to calculate the finite difference $\pi_i \approx \mathsf{R}^ \star_\text{\tiny PCP}(c_i+1) - \mathsf{R}^\star_\text{\tiny PCP}(c_i)$. 
	
\textcolor{blue}
{
	Provided that the $\sum_{l \in L} n_l-1$ choice closing time variables can be found easily, the PCP better suits ranking-choice model contrarily to the CDLP with its $2^n-1$ set duration variables.
}

\textcolor{blue}
{
	Almost all choice models can be turned into a ranking-based choice model by enumerating every choice path. A multinomial logit segment interested in products u, v and w with respective weights 2, 5 and 1 corresponds for example to the tree of Figure~\ref{fig:enumeratingPaths} (weight for quitting is 0) . 
	\begin{figure}[h!]
		\centering
		\begin{tikzpicture}[grow=right, sloped, level 1/.style={sibling distance=10mm},level 2/.style={sibling distance=6mm}]
		\node {S}
		child 
		{
			node {u}        
			child 
			{
				node {v}	
				child 
				{
					node {w}	
					edge from parent 
					node[midway,fill=white]  {\tiny 1.0}				
				}	
				edge from parent
				node[midway,fill=white]  {\tiny 0.83}		
			}
			child {
				node {w}
				child 
				{
					node {v}
					edge from parent
					node[midway,fill=white]  {\tiny 1.0}					
				}	
				edge from parent
				node[midway,fill=white]  {\tiny 0.17}
			}
			edge from parent 
			node[midway,fill=white]  {\tiny 0.25}
		}
		child 
		{
			node {v}        
			child 
			{
				node {u}	
				child 
				{
					node {w}
					edge from parent
					node[midway,fill=white]  {\tiny 1.0}					
				}	
				edge from parent
				node[midway,fill=white]  {\tiny 0.67}		
			}
			child {
				node {w}
				child 
				{
					node {u}
					edge from parent
					node[midway,fill=white]  {\tiny 1.0}					
				}	
				edge from parent
				node[midway,fill=white]  {\tiny 0.33}
			}
			edge from parent 
			node[midway,fill=white]  {\tiny 0.38}
		}
		child 
		{
			node {w}        
			child 
			{
				node {v}	
				child 
				{
					node {u}
					edge from parent
					node[midway,fill=white]  {\tiny 1.0}					
				}	
				edge from parent
				node[midway,fill=white]  {\tiny 0.71}		
			}
			child {
				node {u}
				child 
				{
					node {v}
					edge from parent
					node[midway,fill=white]  {\tiny 1.0}					
				}	
				edge from parent
				node[midway,fill=white]  {\tiny 0.29}
			}
			edge from parent 
			node[midway,fill=white]  {\tiny 0.12}
		};	
		\end{tikzpicture}
		\caption{Choice paths of a multinomial logit segment interested in u, v and w.}
		\label{fig:enumeratingPaths}
	\end{figure}
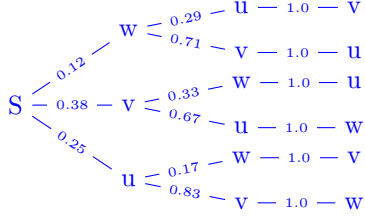	
	The PCP will not be as efficient because the number of choice paths grows exponentially. 
	However, iterative and pruning techniques, that are not presented in this article, can be used to solve this choice models more properly with the PCP. 
}

	\textcolor{blue}
{
	The PCP supports heterogeneous arrival rates by splitting the reservation period and adding constraints with integer variables.
	For example, we can formulate $D^l_{t,k}=\max\left(0, \min(D^k_l,e_t)-s_t\right)$ as the sales duration during period $t$ starting at $s_t \in [0,\tau]$ and ending at $e_t \in [0,\tau]$.
	This multi-period version is thus longer and more complex to solve as for the CDLP where the column generation problem must be adapted to tackle periods.
	We prefer to use the mono-period version with re-optimization as introduced in Section~\ref{productClosing:subsec:staticApproximations} to refine the solution with updated homogeneous arrival rates.
}

\textcolor{blue}
{
	We presented the PB policy in Section~\ref{productClosing:subsec:staticApproximations} with products that are partitioned but in practice they are often nested to protect sales of higher revenue.
	The PCP approach suits with the nesting concept because the respective availability of products can easily be forced.
	For example the linear constraint  $t_{u} \ge t_{v} \ge t_{w}$ ensure the nesting $u > v > w$.
}

\subsection{Reopen the policy} 
\label{productClosing:subsec:reopenThePolicy}

By ranking products by increasing closing time, we define a hierarchy $\mathsf{H} = \{\mathsf{H}_j\}_{j \in J}$ such that $\mathsf{H}_u > \mathsf{H}_v$ if $\mathsf{T}_u > \mathsf{T}_v$ and  $\mathsf{H}_u = \mathsf{H}_v$ if $\mathsf{T}_u = \mathsf{T}_v$.	
If $n_\mathsf{T}$ is the number of distinct closing times then $\mathsf{H}_j \in [1,n_\mathsf{T}]$ is the rank of product $j \in J$, $\mathsf{H}^k \subseteq J$ contains all the products of rank $k \in [1, n_\mathsf{T}]$ and $\mathsf{T}^k_\mathsf{H}$ is the closing time of rank $k \in [1, n_\mathsf{T}]$ products.	 

Any closing times $\mathsf{T}$ has a unique hierarchy and thus has only one equivalent offers duration denoted by $\mathsf{D}(\mathsf{T)}$. In fact there is $n_\mathsf{T}$ non null durations $\mathsf{D}_k = \mathsf{T}_\mathsf{H}^k - \mathsf{T}_\mathsf{H}^{k\text{-}1}$ with $k \in [1, n_\mathsf{T}]$ and $S_k = \bigcup_{\gamma = k}^{n_\mathsf{T}} \mathsf{H}^\gamma$.
So that any PCP solution $\mathsf{T}$ can be evaluated in the CDLP by using $\mathsf{D}(\mathsf{T})$.
 
The durations $\mathsf{D}(\mathsf{T})$ preserve the PCP no-reopening.
The PCP can not use reopening possibility and is thus a lower bound to the CDLP:

\begin{proposition}	
	\label{productClosing:proposition:CDLPUpperBoundOfPCP}
	$\mathsf{R}^\star_{\text{\rm PCP}} \le \mathsf{R}^\star_{\text{\rm CDLP}}$
\end{proposition}

\begin{proof}
	We calculate the CDLP sales for the $\mathsf{D}(\mathsf{T})$ durations:
	\begin{align*}
	 \mathsf{Q}_\text{\tiny CDLP} (\mathsf{D}(\mathsf{T}))
	 	= \sum_{S\subseteq J}\sum_{l \in L} \lambda_l \mathsf{P}_l(j|S) \mathsf{D}(\mathsf{T})_S
	 	= \sum\limits_{l \in L} \lambda_l\sum_{j \in C_l}\sum\limits_{k=1}^{n_\mathsf{T}} \mathsf{P}_l(j|S_k) \mathsf{D}_{k} 
	\end{align*}
	As seen in section \ref{productClosing:subsec:buyinglogicunderclosingpolicy}, each segment buys a product $j \in J$ between $\mathsf{T}_l^{l_j-1}$ and  $\mathsf{T}_l^{l_j}$, with $\mathsf{P}_l(j|S_k)=\mathsf{P}_l(j|\{j\})$ otherwise $\mathsf{P}_l(j|S_k)=0$. These closing times corresponds respectively to $k = \mathsf{H}_{l^{l_j-1}}$ and $k = \mathsf{H}_j$ such that:	
	\begin{align*}
	 \mathsf{Q}_\text{\tiny CDLP} (\mathsf{D}(\mathsf{T}))
	 &	= \sum\limits_{l \in L} \lambda_l\sum_{j \in C_l}\sum\limits_{k= \mathsf{H}_{l^{l_j\text{-}1}}}^{\mathsf{H}_j} \mathsf{P}_l(j|\{j\}) \mathsf{D}_{k} \\
	 &	= \sum\limits_{l \in L} \lambda_l\sum_{j \in C_l} \mathsf{P}_l(j| \{j\}) \left(\mathsf{T}_{\mathsf{H}_j}-\mathsf{T}_{\mathsf{H}_j\text{-}1}+\mathsf{T}_{\mathsf{H}_j\text{-}1}-\dots-\mathsf{T}_{\mathsf{H}_{l^{l_j-1}}}\right) \\
	 &	= \sum\limits_{l \in L} \lambda_l\sum_{j \in C_l} \mathsf{P}_l(j| \{j\})(\mathsf{T}_l^{l_j}-\mathsf{T}_l^{l_j-1}) \\
	 & = \sum\limits_{l \in L} \lambda_l\sum_{j \in C_l} \mathsf{P}_l(j| \{j\})\mathsf{D}_l^k= \mathsf{Q}_\text{\tiny PCP} (\mathsf{T})
	\end{align*}	
	Moreover, $\sum_{S \subseteq J}\mathsf{D}(\mathsf{T})_S \le \mathsf{T}_\mathsf{H}^{n_\mathsf{T}} \le \tau$ such that $\mathsf{D}(\mathsf{T})$ is a CDLP feasible solution and $\mathsf{R}_\text{\tiny CDLP} (\mathsf{D}(\mathsf{T})) = \mathsf{R}_\text{\tiny PCP} (\mathsf{T})$.
	\end{proof}		

One consequence is that we can use the PCP solution as an initial solution for the CDLP. 
We call this approach the Choice Deterministic with Products Closing initial solution (CDPC).

\begin{align*}
	\label{productClosing:eq:CDLP}
	\tag{CDPC}
	\mathsf{R}^\star_\text{\tiny CDPC} (\mathsf{D}^\star) = \; &  \max_{\mathsf{D}} \mathsf{R}_\text{\tiny CDLP} (\mathsf{D})		\\ 
					s.t. \; & \mathsf{D}^0 = \mathsf{D}(\mathsf{T}^\star)\\
							& \mathsf{T}^\star = \arg \max \mathsf{R}_\text{\tiny PCP} (\mathsf{T})
\end{align*}
Where $\mathsf{D}^0$ is the initial solution to solve the CDLP.
CDPC allows to obtain a reopening policy by using the PCP. 
We will also see in the numerical experiments that PCP through the CDPC is useful to accelerate the CDLP for the largest instances.

Another option to reopen solutions, is to use re-optimization thought the reservation period to make the PCP reopen sales if allowed.
If reopening is prohibited, we just have to fix $\mathsf{T}_j=0$ for all the products already closed at the time of re-optimization.

In the case of no-reopening, PCP and CDLP are equivalent:

\begin{proposition}
\label{productClosing:proposition:equivalenceCDLPandPCPIfNoReopening}
If no-reopening,  $\mathsf{R}^\star_{\text{\rm PCP}} = \mathsf{R}^\star_{\text{\rm CDLP}}$
\end{proposition}

\begin{proof}
	if no-reopening, each product has a unique closing time $\mathsf{T}_j$. 
	We denote by $\mathsf{T}(\mathsf{D})$, the products closing times obtained from the no-reopening durations.
	Similarly to the proof of proposition 1, we can show that $\mathsf{R}_\text{\tiny CDLP} (\mathsf{D}) = \mathsf{R}_\text{\tiny PCP} (\mathsf{T}(\mathsf{D}))$ so that  $\mathsf{R}^\star_{\text{\rm PCP}} \ge \mathsf{R}^\star_{\text{\rm CDLP}}$ when no reopening.
	With proposition 1, the result is immediate.
\end{proof}

\section{Solving the PCP} 
\label{productClosing:sec:solvingThePCP}

In this section, we describe how we linearize the PCP to obtain a mixed integer linear program.
We also present methods to rapidly solve the linearization.

\subsection{Benefits of overlapping}
\label{productClosing:subsec:benefitsOfOverlapping}	
	
In the PCP, the only non-linear constraint $\mathsf{T}^k_l = \max_{j \in C^{[k]}_l} \mathsf{T}_j$ determines segment choice closing time.
If segments share a same partial consideration set, for example: $C_{l_1}^{[k]} = C_{l_2}^{[k]}$ with $l_1,l_2 \in L$ , we can eliminate redundant constraints as follows:
 \begin{align*}
	 \mathsf{T}_l^1 = \mathsf{T}_{l^1} 	& \quad \forall l \in L \\	
	 \mathsf{T}_l^k = \mathsf{T}_S \; | \; S = C^{[k]}_l	& \quad \forall k \in [2,n_l], \; l \in L \\	
	 \mathsf{T}_S = \max\limits_{j \in S} \mathsf{T}_j		& \quad \forall S\in C_L
 \end{align*}	
where $\mathsf{T}_S$ is the maximum time of products among offer $S$ and $C_L$ is the union of the segment consideration subsets, determined as follows:
\begin{equation*}
	C_L = \bigcup\limits_{l \in L} \bigcup_{k=2}^{n_l} C_l^{[k]}
\end{equation*}
For example, for two segments with preference lists $l_1:$ $u\xrightarrow{} v \xrightarrow{} w$ and $l_2:$ $v\xrightarrow{} u$, respectively, $C_L$ is $\left\{\{u,v\}, \{u,v,w\}\right\}$. 
We denote by  $n_L$ the cardinality of $C_L$. 
In the previous example $n_L=2 < (n_{l_1}-1) + (n_{l_2}-1) =3$ so that we can use the same constraint to calculate $\mathsf{T}_{l_1}^2$ ($\{u,v\}$) and $\mathsf{T}_{l_2}^2$ ($\{v,u\}$).
The number of non-linear constraint is $n_L$.
It depends on the number of segments, the number of products considered, and the overlap between segments.
A simple analysis allows us to bound $n_L$ between $\max_{l\in L} \;n_l - 1$ when the segments overlap favourably and $\sum_{l\in L} n_l - 1$ when there is no overlap.

\subsection{Linearization}
\label{productClosing:subsec:linearization}

We now linearize the constraint $\mathsf{T}_S = \max_{j \in S} \; \mathsf{T}_j$ with $S \in C_L$ introduced in the previous section \ref{productClosing:subsec:benefitsOfOverlapping}.
Let $S_0$ and $S_1$ be two strict subsets of $S$ ($S_0,S_1 \subset S$) such that $S=S_0 \cup S_1$.
We naturally have $\mathsf{T}_S = \max_{j \in S} \; \{ \mathsf{T}_{S_0}, \mathsf{T}_{S_1}\}$.
We thus only have to compare the closing times of products that $S_0$ and $S_1$ are not sharing to obtain $\mathsf{T}_S$.

Suppose for example that we have $S=\{u,v,w,z\}$, $S_0=\{u,v,w\}$ and $S_1=\{u,z\}$. 
That respects $S=S_0 \cup S_1$.
\textcolor{blue}{We simply have $\mathsf{T}_S=\mathsf{T}_{S_0}$ if and only if $v$ or $w$ are open for longer than $z$.
Conversely, $\mathsf{T}_S=\mathsf{T}_{S_1}$ if and only if $z$ is open for longer than $v$ and $w$.}
The fact that product $u$ is shared by both subsets reduces the number of comparisons.

For any set $S \in C_L$, it always exists two subsets $S_0$ and $S_1$ such that $S = S_0 \cup S_1$.
In fact, it exists $l \in L$, $k \in [1, n_l]$ such that $S = C^{[k]}_l$.
Then $S_0 = C^{[k-1]}_l \subset S$ and $S_1=\{l^k\}  \subset S$ are two strict subsets admissible because $S = S_0 \cup S_1$.

We use the hierarchy defined in section \ref{productClosing:subsec:reopenThePolicy} to `` linearize '' the products comparisons.
We introduce the hierarchy binary variables $\mathsf{H}_{u,v} =1$ if $\mathsf{H}_u \ge \mathsf{H}_v$ and 0 otherwise with $u,v \in J$.

This leads to the PC mixed integer program (PCMP) with the previous definition of $S_0$ and $S_1$:

\begin{align*}
\tag{PCMP}
\label{productClosing:eq:PCMP}
\mathsf{R}^\star_\text{\tiny PCMP} (\mathsf{T}^\star)  = \max_{\mathsf{T}} \quad 
&r^\top \mathsf{Q}_\text{\tiny PCP} (\mathsf{D_L})\\
s.t. \quad
&A \mathsf{Q}_\text{\tiny PCP} (\mathsf{D_L}) \le c\\
& \mathsf{D}_l^k  = \mathsf{T}_l^k - \mathsf{T}^{k\text{-}1}_l							& \forall k \in [1,n_l], \; l \in L \\
& \mathsf{T}_l^1 = \mathsf{T}_{l^1}									& \forall l \in L \\		
& \mathsf{T}_l^k = \mathsf{T}_S \; | \; S = C^{[kø]}_l						& \forall k \in [2,n_l], \; l \in L \\	
&S = S_0 \cup S_1   & \forall S \in C_L, (S_0,S_1) \subset S\\
&\mathsf{T}_S \ge \mathsf{T}_{S_0}  	& \forall S \in C_L, S_0 \subset S\\
&\mathsf{T}_S \ge \mathsf{T}_{S_1}  	& \forall S \in C_L, S_1 \subset S\\
&\mathsf{T}_S \le \mathsf{T}_{S_0} + \tau \sum\limits_{\substack{u \in S_1\setminus S_0 \\ v \in S_0\setminus S_1}} \mathsf{H}_{u,v} 	& \forall S \in C_L,  (S_0,S_1) \subset S\\
&\mathsf{T}_S \le \mathsf{T}_{S_1} + \tau \mathsf{H}_{v,u}  	& \begin{array}{rr}
\forall u\in S_0\setminus S_1, \forall  v\in S_1\setminus S_0\\
\forall S \in C_L, (S_0,S_1) \subset S
\end{array}  \\
&\mathsf{H}_{u,v} \in \{0,1\} 							&\forall u,v \in J
\end{align*}

\textcolor{blue}
{
	The objective function and constraints (1-2) remain the same as in the PCP.
	The constraints (3-4) link the closing time of each segment choice to a set.
	Each set closing time is determined with products closing time by constraints (5-10) as described above.
}

To limit the number of constraints, we must find the strict subsets $S_0 \subset S$ and $S_1 \subset S$ with the highest cardinality.
We do this when building the program, and we exploit the overlap between segments.
Our model uses overlap to eventually reduce complexity.

\subsection{Use of hierarchy} 
\label{productClosing:subsec:useOfHierarchy}

The hierarchy variables represent a hierarchy between products that could be fixed before we solve the PCMP.
This leads to the PC linear program (PCLP) for any fixed hierarchy $\mathsf{H}$:

\begin{align*}
\tag{PCLP}
\label{productClosing:eq:PCLP}
\mathsf{R}^\star_\text{\tiny PCLP}(\mathsf{H}, \mathsf{T}^\star) = \max_\mathsf{T} \quad
& r^\top Q_\text{\tiny PCP} (\mathsf{D}_L)\\
s.t.\quad 
&A Q_\text{\tiny PCP} (\mathsf{D}_L) \le c	&		\\
&\mathsf{D}_l^k  = \mathsf{T}_l^k - \mathsf{T}_l^{k-1}				&\forall k\in [0,n_l], \; l\in L					\\
&\mathsf{T}^k_l = \mathsf{T}_j \; | \; j = \arg \max_{j \in C_l^{[k]}} \mathsf{H}_j					&\forall k\in [0,n_l], \; l\in L
\end{align*} 

For the optimal hierarchy $\mathsf{H}^\star$, the PCLP and PCP are equivalent.
However, there are $n!$ permutations of the products, and each one is an admissible hierarchy.
Determining the optimal hierarchy is thus a difficult combinatorial problem.

It is easier to find a good but not necessarily optimal hierarchy.   
We can for example: 
\begin{itemize}
	\itemsep0em 
	\item Rank products by price;
	\item Rank products by price divided by number of resources;
	\item Reuse a hierarchy from a previous PCMP optimal solution;
	\item Use a hierarchy specified by the company (often called nesting in practice). 
\end{itemize}

Solving the PCLP with a good but not optimal hierarchy allows us to rapidly obtain a good solution by integrating business-specific knowledge in the model.
Depending on the time available, we can use this solution directly as good enough solution or use it to speed up the PCMP branch-and-bound algorithm.

\section{Numerical experiments}
\label{productClosing:sec:numericalExperiments}

In this section, we conduct numerical experiments to benchmark the following approximations:
\begin{description}
	\item[\textbf{CDLP}] \citep{liu2008choice}:
	Described in Section \ref{productClosing:subsec:staticApproximations} and solved by column generation with the \cite{hosseinalifam2016computing} subproblem for preference-list choice behavior. 
	\item[\textbf{SDCP}] \citep{meissner2013enhanced}: 
	We add product constraints for larger subsets until the objective function no longer changes. 
	\item[\textbf{PCLP}:] Presented in Section \ref{productClosing:subsec:useOfHierarchy}.
	The hierarchy is established by ranking products by their price and then by their potential demand if price are equals. The hierarchy is obtained by ranking the products by price.
	\item[\textbf{PCMP}:] 
	Presented in Section \ref{productClosing:subsec:linearization}.
	An initial solution corresponding to the previous PCLP solution is given to the branch-and-bound process, as explained in Section~\ref{productClosing:subsec:useOfHierarchy}.
	The relative integrability gap is set to $10^{-3}$.
	\item[\textbf{CDPC}:] 
	The CDLP approximation with an initial solution given by the PCMP, as explained in Section~\ref{productClosing:subsec:reopenThePolicy}.
\end{description}

\noindent We use the following policies:
\begin{description}
	\item[\textbf{OP}] 
	is the OP policy described in Section~\ref{productClosing:subsec:staticApproximations}.
	It is obtained by a lexicographic sequencing of the CDLP durations $\mathsf{D}$.
	\item[\textbf{PB}] 
	is the PB policy corresponding to fixing a static limit $q_j$ for each product, as explained in Section~\ref{productClosing:subsec:staticApproximations}.
	\item[\textbf{PC}] 
	is the PC policy returned by the PCP, as explained in Section~\ref{productClosing:sec:closingApproximation}.
	\item[\textbf{OD}] 
	is the OD policy described in Section \ref{productClosing:subsec:dynamicApproximations}.
	It is obtained by the dynamic decomposition of \cite{bront2009column} with $\beta = 1$.
\end{description}
\textcolor{blue}
{
	We use the CDLP, SDCP and CDPC models without any additional constraints.
	They thus allow reopening contrarily to the PCMP and PCLP models that naturally prevent it (see Section~\ref{productClosing:sec:closingApproximation}).
}

\textcolor{blue}
{
\noindent All the notations used in this section are the following:
\begin{description}
	\item[\textbf{R}] is the revenue corresponding to the approximation optimal solution. 
	\item[\textbf{CPU}] is the seconds needed to solve the approximation.
	\item[\textbf{E[R]}] is the expected revenue of the approximation policy.
	We use a Monte-Carlo approach with a discrete-arrival simulation to determine the expected revenue.
	We generate random discrete arrivals by generating arrival timings according to a Poisson process for each segment.
	Each simulation is stopped after a number of evaluations specific to the instance.
	\item[$\bm\Delta$\textbf{E[R]}] is the  expected revenue relative difference between the approximation and the CDLP with OP policy.
	\item[\textbf{LF}] is the load factor equals to $\sum_{l \in L} \lambda_l / \sum_{i \in I} c_i$. It is simply the sum of arrivals over the sum of capacities. We build scenarios by varying the load factor. By multiplying all the $\lambda_l$ by the same factor, we obtain the desired LF.
	\item[\textbf{CF}] is the capacity factor as the percentage of remaining capacity and is thus equal to $\sum_{i \in I} x_i / \sum_{i \in I} c_i$.
\end{description}
}
	
\textcolor{blue}
{
	Because instances are too large for calculating the optimal revenue $\mathsf{V^{\star}}$, we select the best approximations by sorting them by highest expected revenue (E[R]) and lowest running time (CPU).
}

\subsection{Parallel flights} 
\label{sec:sub:parallelFlights}

Our first instance, parallel flights, is illustrated in Figure \ref{fig:productClosing:parallelFlightsResources}. 
It is composed of three parallel flights, of capacity 100, from city A to city B at 09:00, 11:00, and 20:00.
\begin{figure}[h!]
\centering
\includegraphics[max width=\textwidth]{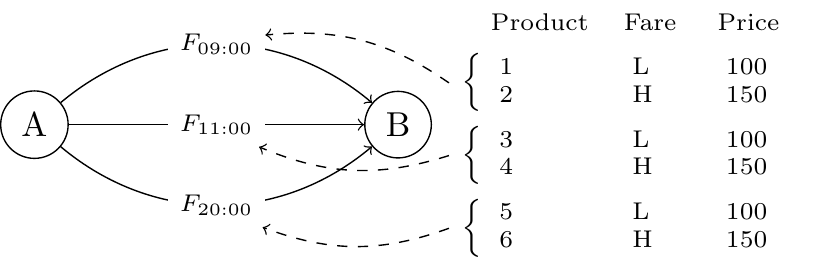}
\caption{Resources and products for parallel flights. \label{fig:productClosing:parallelFlightsResources}}
\end{figure}
We consider two fares $H$ (150) and $L$ (100) per flight, giving six products.
The reservation period lasts 360 periods.
The customers are divided into four segments, as shown in Table~\ref{tab:productClosing:parallelFlighsSegments}.
\begin{table}[h!]
\centering
\includegraphics[max width=\textwidth]{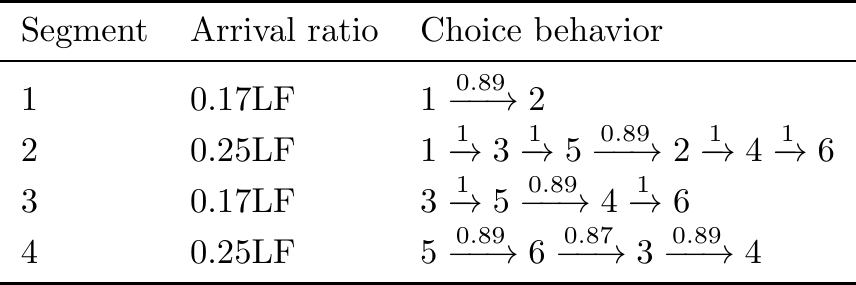}
\caption{Segments for parallel flights. \label{tab:productClosing:parallelFlighsSegments}}
\end{table}

Table \ref{tab:parallelFlightsResults} presents the running time and expected revenue for the parallel flights instance.
We first note that approximations return very similar results for a same policy.
It means that the three approximations return similar solutions as we can see in Table \ref{productClosing:ec:tab:parallelFlights} of the e-companion where approximations share the same approximation revenue and capacity factor.
	\textcolor{blue}
{
	CDLP is an upper bound for PCP in the general case. The same objective value between CDLP and PCP means that reopening does not help to increase the revenue. It is only because of the instance structure. For most of the instances we generate, there is often no advantage by using reopening such that CDLP and PCP return the same approximation revenue.
}

If we now focus on policies, we observe that OD almost always performs better than others in terms of expected revenue.
In average, it is 1.4\% better than the CDLP-OP reference whereas PB and PC are respectively -0.3\% and 0.1\%.
In fact, it is the only dynamic policy and it takes into account the order of arrivals contrarily to the three other static policies OP, PB and PC.
We can also see the effect of the dynamic aspect in Figure \ref{fig:parallelFlightsRelativeCapacityFactor} where the OD policy often has the highest expected capacity factor meaning that it captures more bookings.
It also explains why the OD policy performs better in comparison with other policies when the load factor is low or high.

\begin{table}[h!]
\centering
\includegraphics[max width=\textwidth]{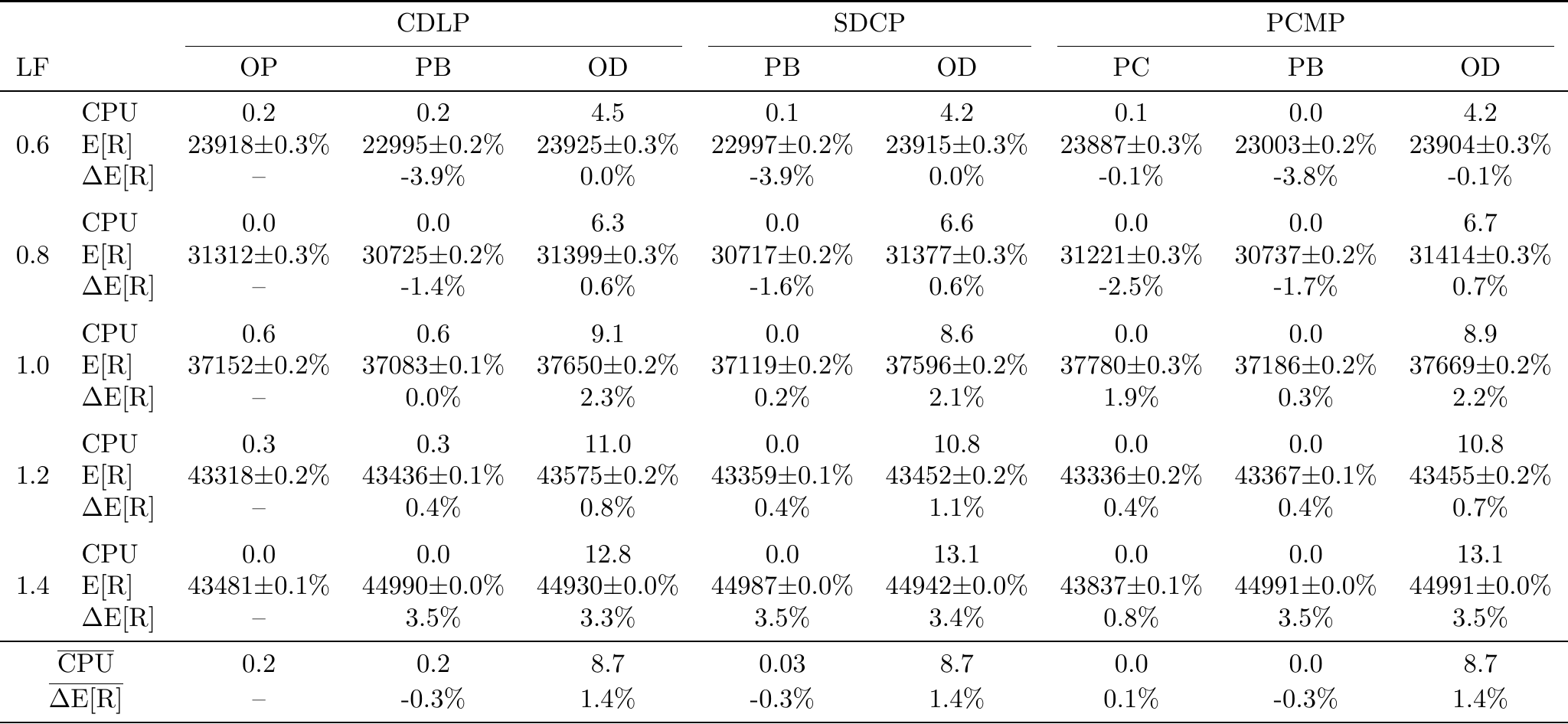}
\caption{Results for Parallel flights with a 3000 evaluations simulation.\label{tab:parallelFlightsResults}}
\end{table}

However, when we compare the running time, the OD policy is by far the slowest whereas OP, PC and PB are equivalent.
The latest policy are in average 30 times faster than the OD policy for this instance.
This long running time comes from its building process as we can see in the Table~\ref{productClosing:ec:tab:parallelFlights} of the e-companion where the time for building each policy is reported. 
This is mainly due to the high number of dynamic program to solve as we explained in Section \ref{productClosing:subsec:dynamicApproximations}.
Moreover, this building time increases with the load factor because it depends on the number of arrivals, the capacities and the number of resources.

\begin{figure}[h!]
\centering
\includegraphics[max width=\textwidth]{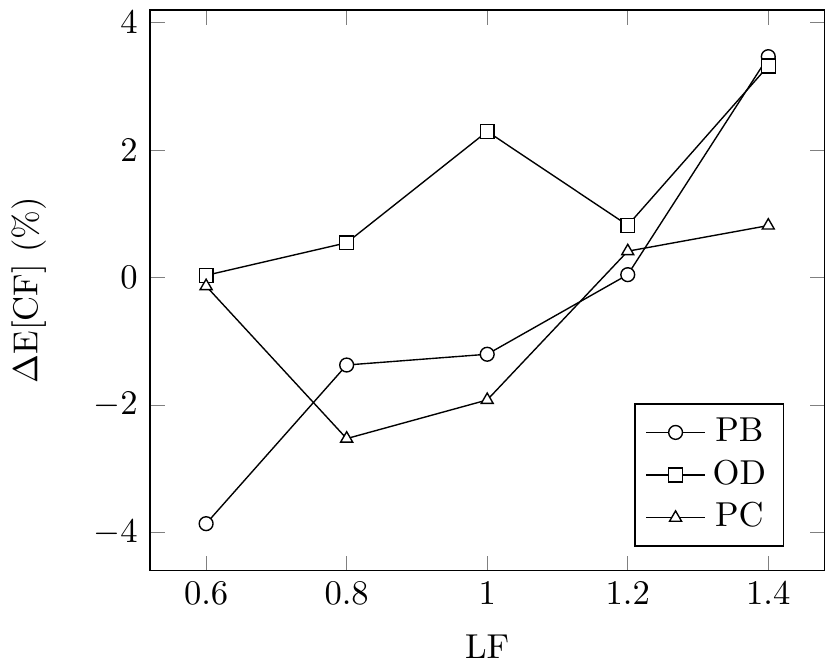}
\caption{Expected capacity factor relative difference $\Delta$E[CF] with respects to CDLP-OP for Parallel flights. \label{fig:parallelFlightsRelativeCapacityFactor}}
\end{figure}

This example also highlights the really unequal performance of the PB policy with respect to the load factor.
It is outperformed by the CDLP-OP for LF inferior or equal to one but up to 3\% better for higher load factor.
This is due to the fact that PB policy capture exactly the number of bookings provided by the related approximation.
Such that when the load factor is inferior to one, it will never capture any eventual additional demand even if capacities are not reached.
It also explains that the capacity factor is really low when the load factor is inferior to one in Figure \ref{fig:parallelFlightsRelativeCapacityFactor} and in comparison of the other policies.
Nonetheless, when the load factor is up to one, the PB policy becomes a really efficient policy because capacities are reached in the approximation.

One important fact regarding the SDCP approximation is that it cannot return an OP policy even if it is built on offers duration.
In fact, the products constraints added, as explained in \cite{meissner2013enhanced}, do not ensure homogenized durations across segments. 
For $LF = 1$, the second segment offers $\{3,5\}$ and $\{1,3,5\}$ respectively during $89.4$ and $270.4$ periods while third segments offers $\{3,5\}$ during $360$ periods.
Products constraints are respected but we cannot conclude offers duration shared by every segment.
That is why the SDCP solution is only used to build PB and OD policy for numerical experiments.

\subsection{Bus-line instance} 
\label{sec:productsClosing:numericalExperiments:BusLine}

The bus-line instance has two buses leaving at 07:00 and 11:00 from city A to cities B, C, and D.
\begin{figure}[h!]
\centering
\includegraphics[max width=\textwidth]{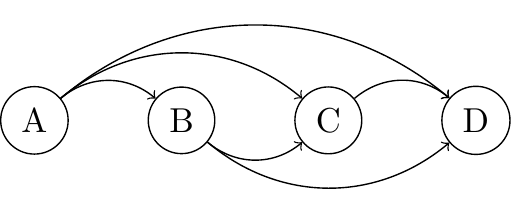}
\caption{Markets for Bus-line instance.	\label{fig:productClosing:busLineMarkets}}
\end{figure}
Six markets are thus served, as illustrated in Figure \ref{fig:productClosing:busLineMarkets}.
Each bus has a capacity of 30 and there are $2 \times 3 = 6$ resources.
Two fares (low, high) are offered for each trip, giving a total of $6 \times 2 \times 2 = 24$ products.
In the bus industry, tickets are usually available at least two months in advance, so we set $T=60$ days.
We consider five segments each considering 4 products.
In total there are $3 \times 6 = 18$ segments.
A complete description of the instance is given in the e-companion at \ref{productClosing:ec:tab:instances}.

Table \ref{tab:buslineResults} shows the running time and expected revenue for the Bus-line instance.
We come to the same conclusions as for the previous Parallel flights instance concerning the equivalence of approximations.
We note that the performance of the PB, PC and OD policies over the CDLP-OP improves as load factor increases.
For the PB policy, the reason is the same as for the Parallel flights instance and is explained in Section~\ref{sec:sub:parallelFlights}.
PC is a more robust policy than OP when there is nonreopening. 
The dynamic aspect of OD ensures better expected revenue than other policies.
These respective qualities of PC and OD are emphasized when the load factor increases because the policy is more selective contrarily to a low load factor for which most of the demand is accepted.

\begin{table}[h!]
\centering
\includegraphics[max width=\textwidth]{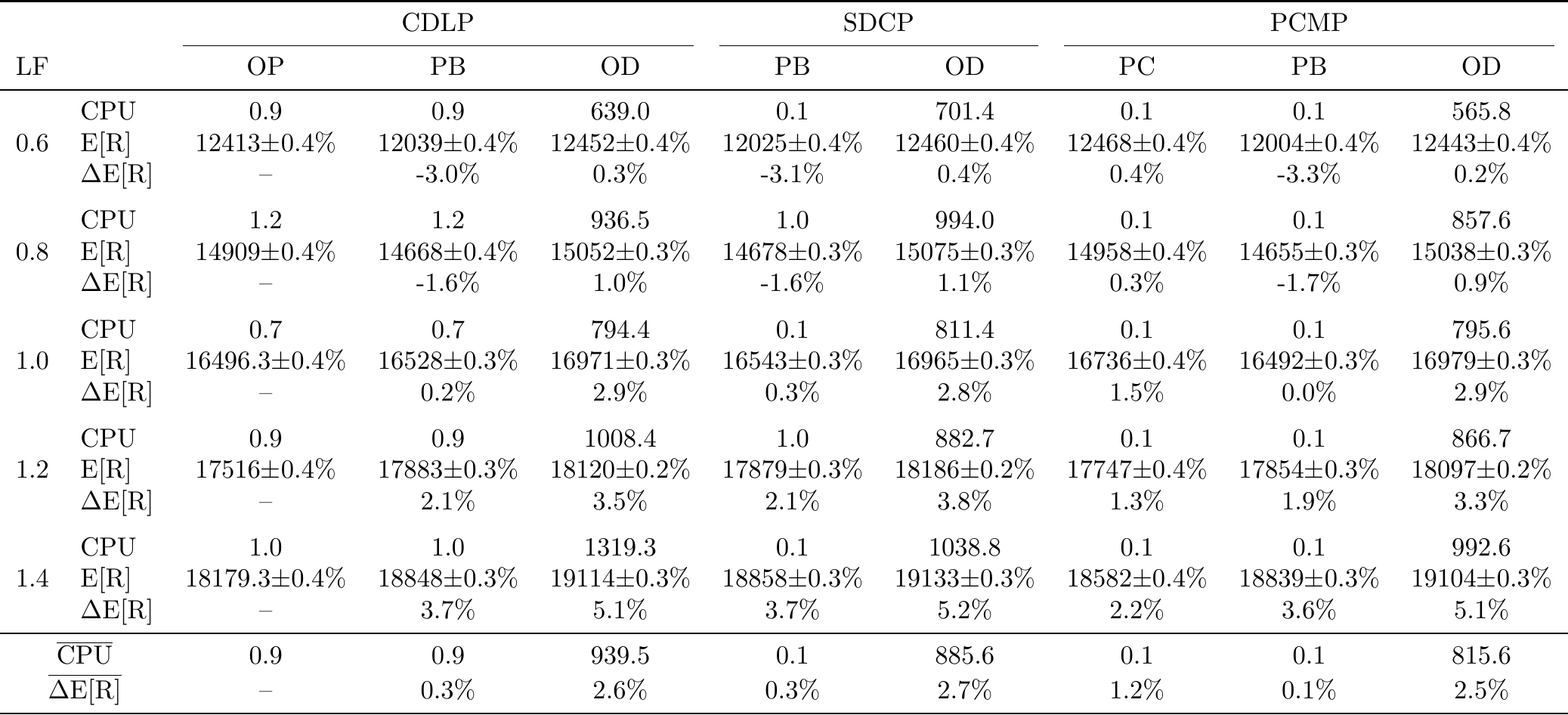}
\caption{Results for Bus-line with a 1000 evaluations simulation. \label{tab:buslineResults}}
\end{table}

This example underlines the good performance of the SDCP which is solved in average 5 times faster than the CDLP.
The products closing constraints added are sufficient to return the same \textcolor{brown}{approximation revenue} as reported in Table~\ref{productClosing:ec:tab:busLine} of the e-companion.
We cannot build OP policy but the policies PB and OD derived perform as well as or better than the CDLP ones for any load factor.

\textcolor{blue}
{
	 We note that the time to solve m times the PCMP is always faster than solving one CDLP. We can generalize this result to all the numerical experiments instances. It ensures that our approach can be used in a dynamic approximation with a finite difference calculus (see Section~\ref{productClosing:subsec:productClosingProgram}) at least as fast as the CDLP.   
}

It is clear in Table \ref{tab:buslineResults} that building the OD policy requires important postprocessing, as explained in Section~\ref{productClosing:subsec:dynamicApproximations}, and thus considerable time. 
Table \ref{productClosing:ec:tab:busLine} confirms that almost all the running time is spent on building the policy and not in solving the approximation.
Even if a leg decomposition is used, a mathematical program must be solved per leg $i \in I$ for each remaining capacity $c_i$ and each potential arrival $ \sum_{i \in I} c_i \times \text{LF}$.
Therefore, the number $N_\text{OD}$ of values to find and store for the OD policy is: 
\begin{equation*}
N_{\text{OD}}= \sum_{i \in I}c_i \left( \sum_{l\in L} \lambda_l \tau \right)
\end{equation*}
The bus-line instance is relatively small, but $N_\text{OD}$ is already equal to $6 \times 30 \times (6 \times 30LF) = 32400 LF$.
It explains why the running time increases when the load factor augments as observed at Table \ref{tab:buslineResults}.

To investigate the OD tractability, we complicate the initial instance progressively and report the number of values $N_\text{OD}$ and the time needed to build this policy at Table \ref{tab:busLineTractability}.
\begin{table}[h!]
\centering
\includegraphics[max width=\textwidth]{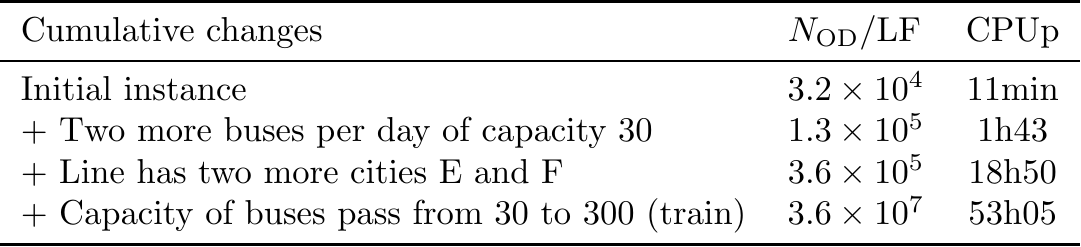}
\caption{Time CPUp to build the OD policy for Bus-line cumulative changes. \label{tab:busLineTractability}}
\end{table}
The OD policy is without doubts the best but become rapidly intractable when instances grow.
Each value to find is often obtained by solving a complex model as explained in Section~\ref{productClosing:subsec:dynamicApproximations}.
And also because computationally it is a lot of values to store.
In practice, the reservation systems may not support this amount of data for a complete network.

\textcolor{blue}
{
	In order to take into account the stochastic aspect of the demand without solving any dynamic approximation, we use re-optimization through the reservation period as introduced in Section~\ref{productClosing:subsec:staticApproximations}.
	We split the reservation period in 5, 10 and 20 checkpoints.
	When calculating the expected revenue, we stop the simulation at each one of these checkpoints to re-optimize.
	The policy obtained at each checkpoint is then valid only until the next checkpoint.
	We plot at Figure~\ref{fig:reoptimization} a comparison between CDLP and PCMP (complete results are available at Table \ref{productClosing:ec:tab:reopt}).
}

	\begin{figure}[h!]
		\centering
	\includegraphics[max width=\textwidth]{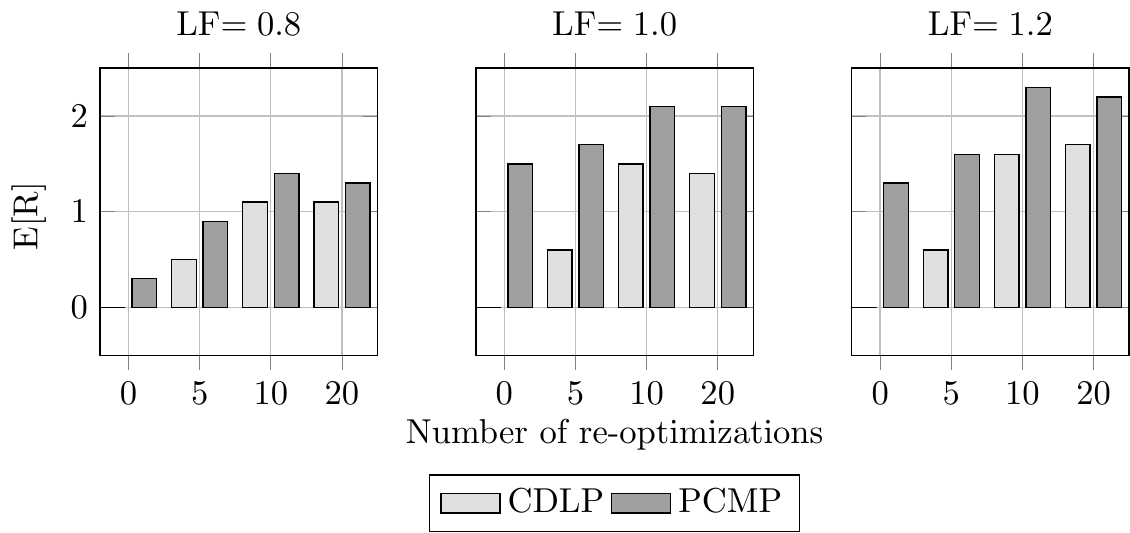}
		\caption{Effect of re-optimization on the Bus line instance.}
		\label{fig:reoptimization}
	\end{figure}

\textcolor{blue}
{
	We note that the re-optimization increases the expected revenue of these two approximations, up to 1.6\%. 
	The number of re-optimizations improves results but seems to reach a maximum with 10 re-optimizations.
	The gap with the OD policy is reduced even if this policy is still between 0.3\% and 0.8\% better.
}

\subsection{Airline instance} 
\label{sec:sub:airline}

The airline instance is based on the Delta Air Lines network limited to eight major US airports, as illustrated in Figure \ref{fig:airlineMarkets}.
We start by limiting the instance on the five largest airports: ATL, LAX, ORD, DFW, and DEN.
A complete description of the instance is given in the e-companion at \ref{productClosing:ec:tab:instances}.
\textcolor{blue}
{
	Segments were generated with different considerations (price, flight duration, departure time, arrival time, product conditions, direct flight) in order to represent real customer behaviors and to have many different permutations of products.
}

\begin{figure}[h!]
\centering
\includegraphics[max width=\textwidth]{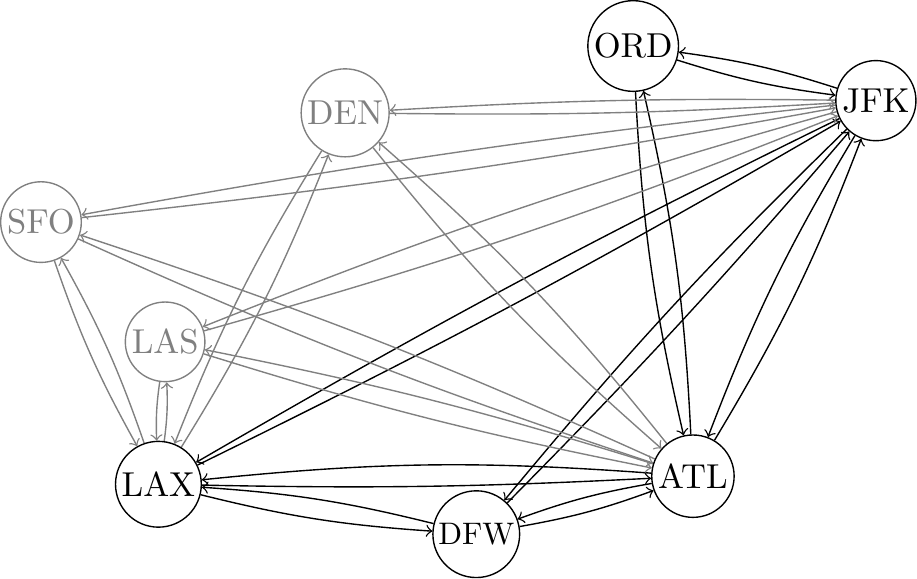}
\caption{Markets of Airline. The five largest airports are represented in bold. \label{fig:airlineMarkets}}
\end{figure}

We do not benchmark the OD policy for this instance because the problem become intractable for this size, as shown in Section~\ref{productClosing:subsec:dynamicProgramFormulation} and confirmed by tests.
For the SDCP, the number of products constraints is at most ${{L}\choose{2}} \times 2^{\max\limits_{l\in L}n_l}= 95703 \times 1024 \approx 9.8 \times 10^7$ according to \cite{meissner2013enhanced}.
Even if this is an upper bound, the search for the intersections between segments is intractable.
That is why we do not benchmark the SDCP in the Airline instance.
The CDLP with column generation takes much time to solve and PCMP resolution is more difficult.
We thus introduce the CDPC and PCLP approximations for this larger instance.

\begin{table}[h!]
\centering
\includegraphics[max width=\textwidth]{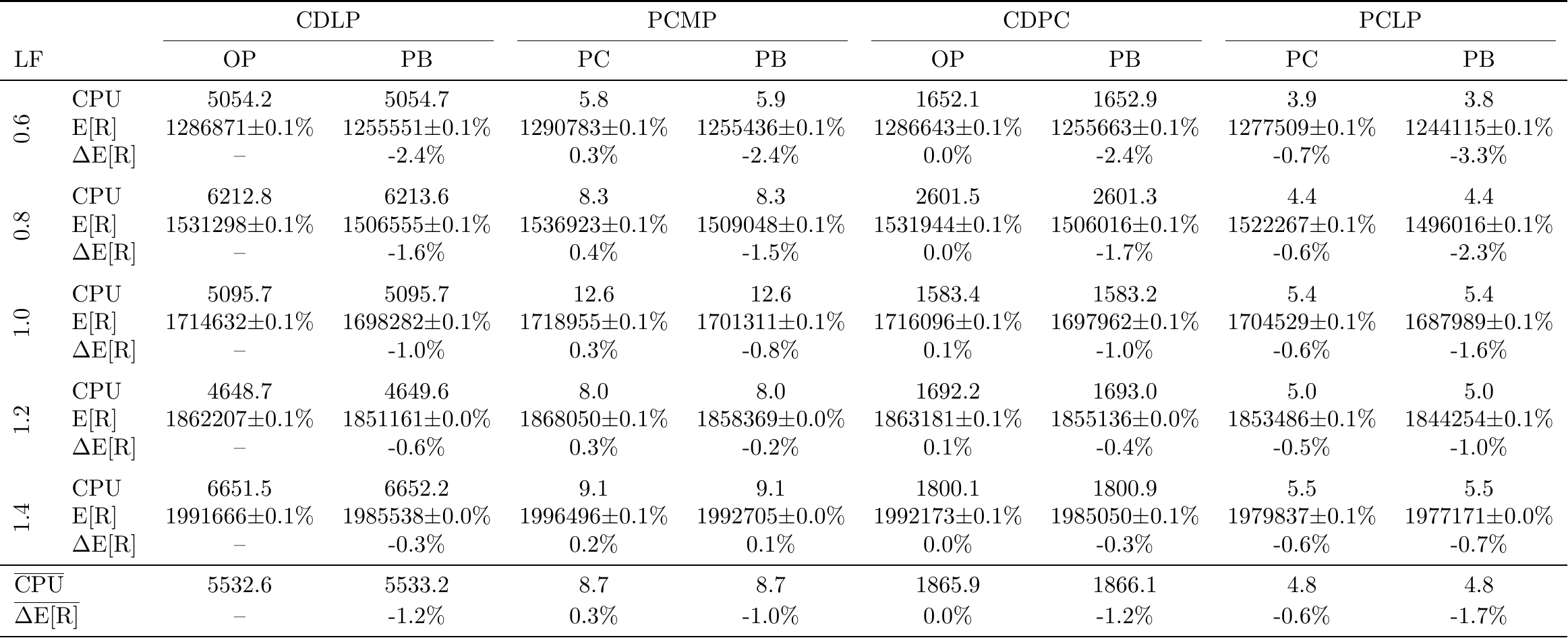}
\caption{Results for Airline with a 500 evaluations simulation. \label{tab:airlineResults}}
\end{table}

Table \ref{tab:airlineResults} reports the running time and expected revenues of the CDLP, PCMP, CDPC and PCLP for the Airline instance with different load factor. 
The full results are reported in Table \ref{productClosing:ec:tab:airline} of the e-companion.

We observe the same phenomenon for the PB policy as for the previous instances.
It can not capture the excess of the demand which is problematic for low factor and rapidly overshadowed by capacity saturation when load factor increases.

We also note that our approach is computed in less than 15~seconds, which is remarkable given the instance size.
It is much faster than the CDLP and always returns a slightly better expected revenue.
This gain in revenue, in average 0.3\%, for the PCMP must be explained by the robustness of closing sales once rather than proposing different offers over the reservation period.

Even though, we note that the CDLP always returns a slightly better optimal revenue in the e-companion at  \ref{productClosing:ec:tab:airline}.
This may be explained by the integrity gap chosen for PCMP or the reopening permitted by CDLP.

\begin{figure}[h!]
\centering
\includegraphics[max width=\textwidth]{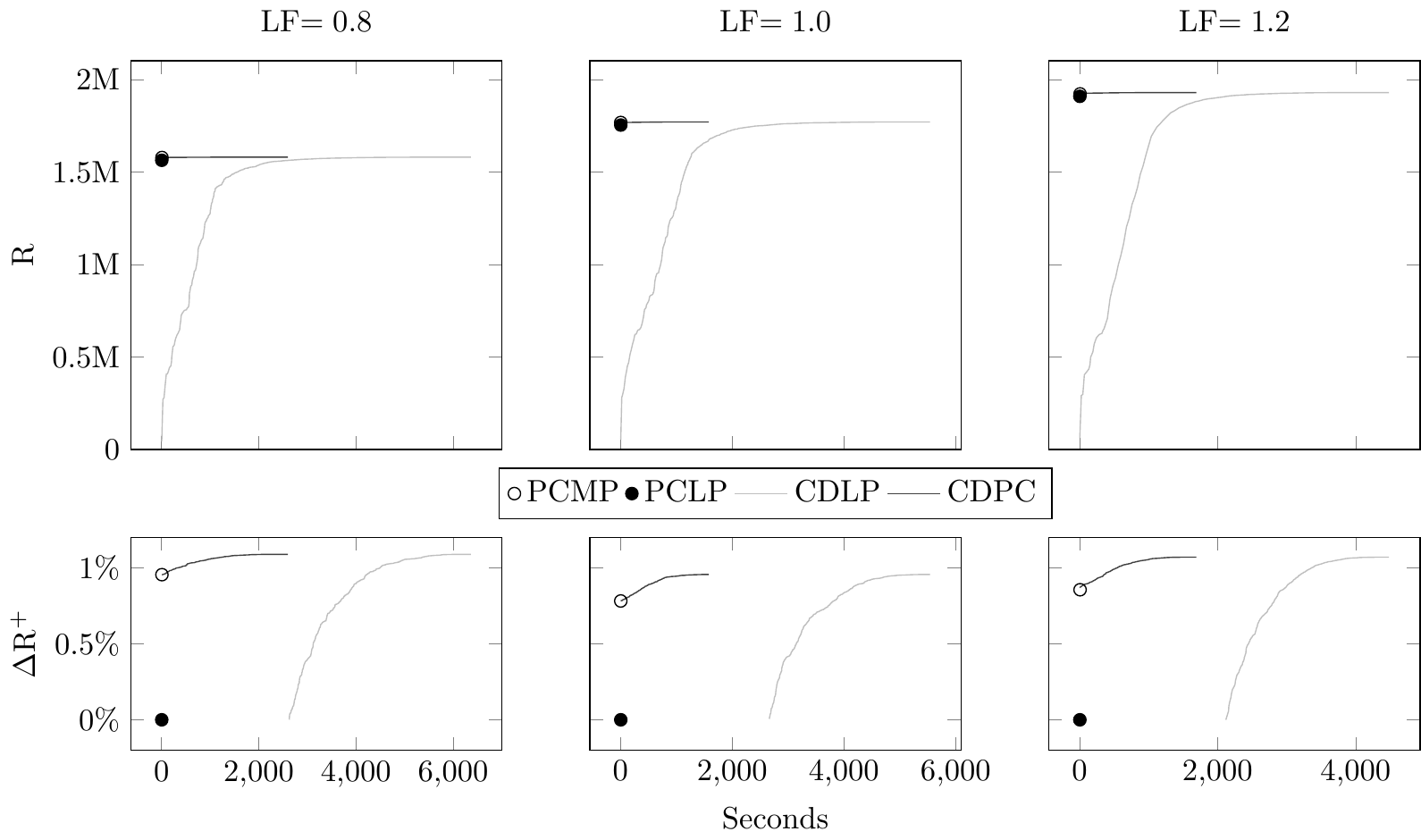}
\caption{\textcolor{blue}{Approximation revenue} towards time for Airline.\label{fig:airlineConvergence}}
\end{figure}

This instance also shows the good quality of our PCLP heuristic. 
In fact PCLP is solved twice faster than PCMP and returns an expected revenue only $0.59\%$ lower than the CDLP-OP.	
However, solving PCMP remains quick and the difference in expected revenue with this approximation is almost $1.0\%$.

We also observe the good performance of our CDPC approach.
It accelerates in average by three the CDLP resolution and returns the same ideal revenue (see e-companion Table \ref{productClosing:ec:tab:airline}) and similar expected revenue, as we can see in Table \ref{tab:airlineResults}, with a 0.04\% difference.
We thus obtain in much less time a really good reopening solution by mixing PCMP and CDLP.

To better illustrate the convergence speed, we plot in Figure \ref{fig:airlineConvergence} the \textcolor{blue}{approximation revenue} R of each approximation vs.~the solution time for different load factor.
CDLP and CDPC are plot by cherry piking and smoothing their column generation solving. 
$\Delta$R$^+$ is the \textcolor{blue}{approximation revenue} relative difference in percent with respect to the PCLP when positive.

We observe that our PCMP approximation rapidly returns a \textcolor{blue}{very good solution} contrarily to the CDLP.
The latter takes more than one hour to converge to solutions found in average in less than 15\,s by PCMP.

The gain in time by choosing the PCMP as a initial solution for the CDLP is perfectly represented in the Figure \ref{fig:airlineConvergence}.
We note that the remaining column generation increases only by less than 0.1\% the solution and the convergence is very slow.

To test the tractability of our approach, we now increase progressively the number of cities in the network.
Table \ref{tab:airlineScalability} lists the evolution of the network characteristics.

\begin{table}[h!]
\centering
\includegraphics[max width=\textwidth]{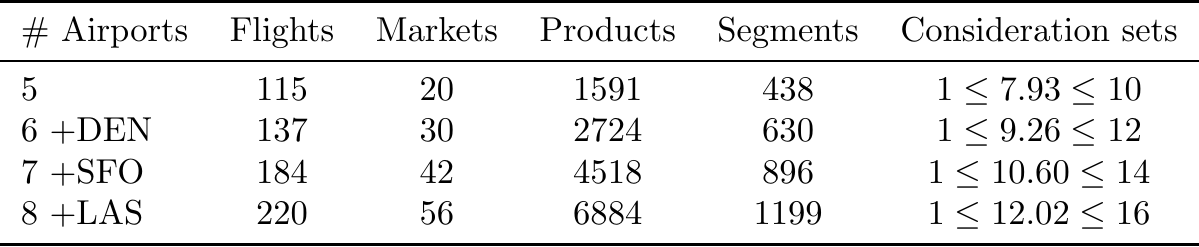}
\caption{Airline characteristics by number of cities considered. The five initial cities are ATL, LAX, ORD, DFW, and DEN. \label{tab:airlineScalability}}
\end{table}

Figure \ref{fig:airlineScalability} reports the running time CPU, on a logarithmic scale, for CDLP, PCMP CDPC and PCLP and the expected revenue E[R] for different sizes of network.

\begin{figure}[h!]
\centering
\includegraphics[max width=\textwidth]{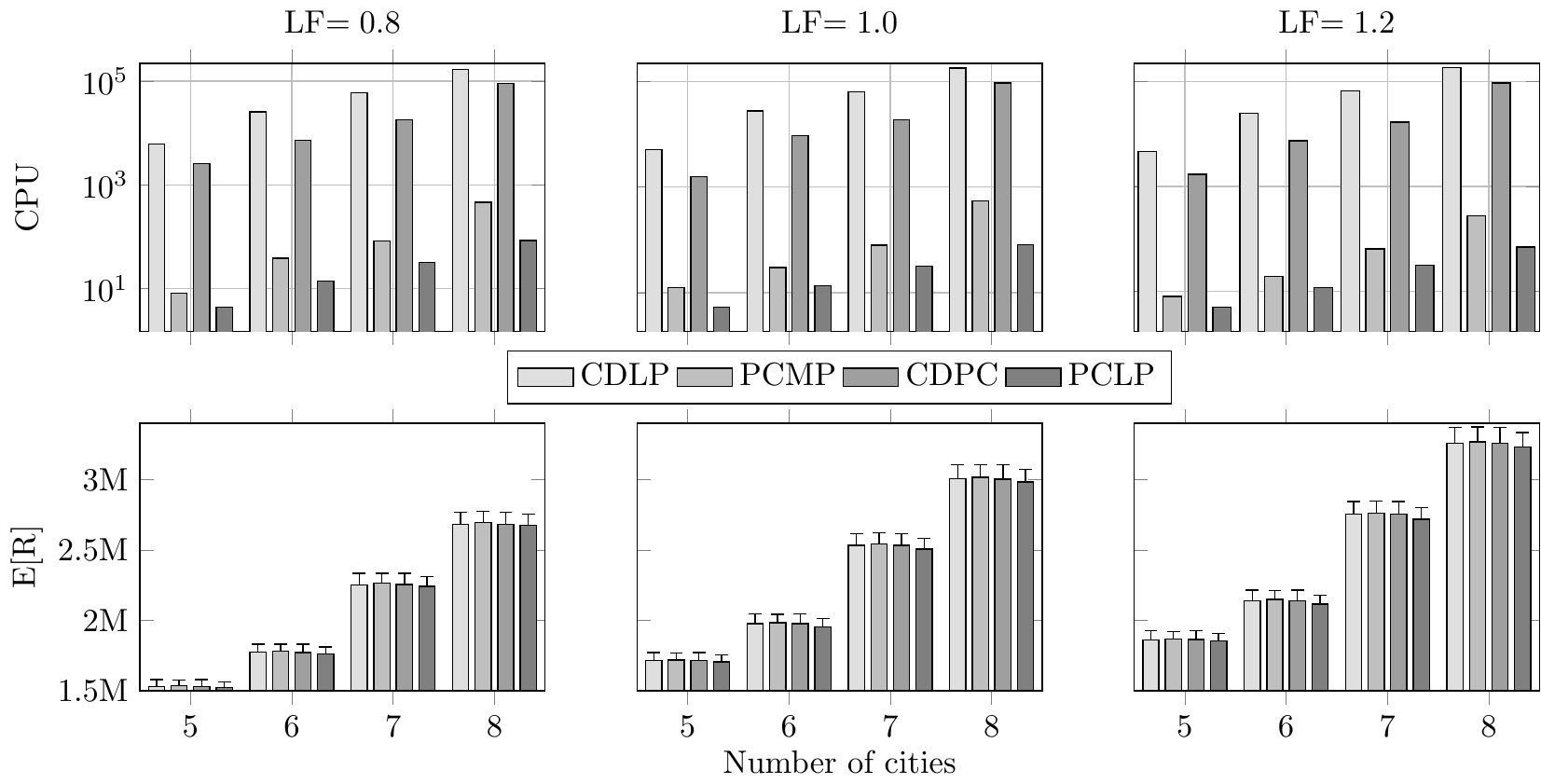}
\caption{Running seconds CPU and expected revenue E[R] for Airline with scaled number of cities considered.\label{fig:airlineScalability}}
\end{figure}

The running times are really similar for the load factors experimented.
The faster resolution of the PCMP in comparison with the CDLP is even more pronounced as the network grows.
Indeed, the CDLP is far longer to solve because each subproblem highly suffers from the increase of products.

The difference between the CDLP and the CDPC running time is considerable.
In fact, it corresponds to the time for the CDLP to reach the PCMP ideal revenue.
This shows how much the PCMP convergence (branching the hierarchy binaries) is faster than the CDLP column generation.
Moreover, it emphasizes the significant benefice of taking the PCMP as initial solution for the CDLP (CDPC).

Not surprisingly, the expected revenue is higher as the load factor or the number of cities increases.
We note that the PCMP returns a slightly better expected revenue (between 0.3\% and 0.6\%).
As for the previous instances, this illustrates the more robust structure of the PC policy.

We observe that, in average, the PCMP is solved in 60\,s for 7 cities and in 450\,s for 8 cities.
This noticeable gap underlines the first difficulties for the PCMP as instances grows.
On another side, the PCLP requires respectively 38\,s and 60\,s and does not seem not as impacted by this scaling.
Its solving time increases smoothly and the expected revenue is only respectively 0.3\% and 0.6\% lower than CDLP and PCMP.
The PCLP seems a good alternative for largest instance and the expected revenue returned could be improved by better method to select the hierarchy.

\section{Conclusion}
\label{productClosing:sec:conclusion}

We have presented a new static approximation for the CNRM problem with \textcolor{blue}{ranking-based} choice behavior.
We focus on the preference list because the multinomial logit model suffers from the independence of irrelevant alternatives.
Rather than working with offers, we work directly with the products and determine when to stop selling each one.
For small and medium instances, the different approximations and associated policies (OP, PC, PB, OD) give similar results.
However, OD can give the best results if the leg decomposition is appropriate for the instance, because of its dynamic adaptation to the stochastic demand.
For larger instances, our approximation outperforms the current approximations because the policy gives a slightly better expected revenue for a much shorter solution time.
Our approximation is based on a no-reopening policy. A solution with reopening can be generated by using the  PCMP solution as an initial solution for CDLP.
This two-phase approach greatly accelerates CDLP. 
For even larger instances, our approximation is designed to become linear if a hierarchy is fixed.
A good hierarchy is in practice not hard to find.
The linear program obtained can be rapidly solved and returns a near-optimal solution.
With its greatly reduced solution time and good-quality policy, our approximation is a promising approach for practical implementations.  

\section*{Acknowledgments}

The authors are grateful to the natural sciences and engineering research council of Canada (NSERC), the Fonds de recherche du Quebec en nature et technologies (FRQNT) and the company ExPretio technologies for having funded and supported this research.
The authors also thank the anonymous referees and associate editor, whose feedback helped improve the paper.

\section*{References}
\bibliographystyle{elsarticle-harv}
\bibliography{bibliography}

\begin{thebibliography}{25}
\expandafter\ifx\csname natexlab\endcsname\relax\def\natexlab#1{#1}\fi
\expandafter\ifx\csname url\endcsname\relax
  \def\url#1{\texttt{#1}}\fi
\expandafter\ifx\csname urlprefix\endcsname\relax\def\urlprefix{URL }\fi

\bibitem[{Ben-Akiva and Lerman(1985)}]{ben1985discrete}
Ben-Akiva, M., Lerman, S.~R., 1985. Discrete Choice Analysis: Theory and
  Application to Travel Demand. Vol.~9. MIT Press.

\bibitem[{Bront et~al.(2009)Bront, M{\'e}ndez-D{\'\i}az, and
  Vulcano}]{bront2009column}
Bront, J. J.~M., M{\'e}ndez-D{\'\i}az, I., Vulcano, G., 2009. A column
  generation algorithm for choice-based network revenue management. Operations
  Research 57~(3), 769--784.

\bibitem[{Chaneton and Vulcano(2011)}]{chaneton2011computing}
Chaneton, J.~M., Vulcano, G., 2011. Computing bid prices for revenue management
  under customer choice behavior. Manufacturing and Service Operations
  Management 13~(4), 452--470.

\bibitem[{Erdelyi and Topaloglu(2010)}]{erdelyi2011using}
Erdelyi, A., Topaloglu, H., jan 2010. Using decomposition methods to solve
  pricing problems in network revenue management. Journal of Revenue and
  Pricing Management 10~(4), 325--343.

\bibitem[{Farias et~al.(2013)Farias, Jagabathula, and
  Shah}]{farias2013nonparametric}
Farias, V.~F., Jagabathula, S., Shah, D., 2013. A nonparametric approach to
  modeling choice with limited data. Management Science 59~(2), 305--322.

\bibitem[{Gallego et~al.(2004)Gallego, Iyengar, Phillips, and
  Dubey}]{gallego2004managing}
Gallego, G., Iyengar, G., Phillips, R., Dubey, A., 2004. Managing flexible
  products on a network. Tech. Rep. TR-2004-01, Columbia University.

\bibitem[{Gallego et~al.(2014)Gallego, Ratliff, and
  Shebalov}]{gallego2014general}
Gallego, G., Ratliff, R., Shebalov, S., 2014. A general attraction model and
  sales-based linear program for network revenue management under customer
  choice. Operations Research 63~(1), 212--232.

\bibitem[{Hanson and Martin(1996)}]{hanson1996optimizing}
Hanson, W., Martin, K., jul 1996. Optimizing multinomial logit profit
  functions. Management Science 42~(7), 992--1003.

\bibitem[{Hosseinalifam et~al.(2016)Hosseinalifam, Savard, and
  Marcotte}]{hosseinalifam2016computing}
Hosseinalifam, M., Savard, G., Marcotte, P., 2016. Computing booking limits
  under a non-parametric demand model: A mathematical programming approach.
  Journal of Revenue and Pricing Management 15, 170--184.

\bibitem[{Jagabathula and
  Rusmevichientong(2017)}]{jagabathula2017nonparametric}
Jagabathula, S., Rusmevichientong, P., sep 2017. A nonparametric joint
  assortment and price choice model. Management Science 63~(9), 3128--3145.

\bibitem[{Kunnumkal and Topaloglu(2010)}]{kunnumkal2010new}
Kunnumkal, S., Topaloglu, H., 2010. A new dynamic programming decomposition
  method for the network revenue management problem with customer choice
  behavior. Production and Operations Management 19~(5), 575--590.

\bibitem[{Liu and van Ryzin(2008)}]{liu2008choice}
Liu, Q., van Ryzin, G., 2008. On the choice-based linear programming model for
  network revenue management. Manufacturing and Service Operations Management
  10~(2), 288--310.

\bibitem[{Meissner et~al.(2013)Meissner, Strauss, and
  Talluri}]{meissner2013enhanced}
Meissner, J., Strauss, A., Talluri, K., apr 2013. An enhanced concave program
  relaxation for choice network revenue management. Production and Operations
  Management 22~(1), 71--87.

\bibitem[{Paul et~al.(2018)Paul, Feldman, and Davis}]{paul2018assortment}
Paul, A., Feldman, J., Davis, J.~M., 2018. Assortment optimization and pricing
  under a nonparametric tree choice model. Manufacturing \& Service Operations
  Management 20~(3), 550--565.

\bibitem[{Ratliff et~al.(2008)Ratliff, Rao, Narayan, and
  Yellepeddi}]{ratliff2008multi}
Ratliff, R.~M., Rao, B.~V., Narayan, C.~P., Yellepeddi, K., 2008. A
  multi-flight recapture heuristic for estimating unconstrained demand from
  airline bookings. Journal of Revenue and Pricing Management 7~(2), 153--171.

\bibitem[{Rusmevichientong et~al.(2014)Rusmevichientong, Shmoys, Tong, and
  Topaloglu}]{rusmevichientong2014assortment}
Rusmevichientong, P., Shmoys, D., Tong, C., Topaloglu, H., 2014. Assortment
  optimization under the multinomial logit model with random choice parameters.
  Production and Operations Management 23~(11), 2023--2039.

\bibitem[{Strauss et~al.(2018)Strauss, Klein, and
  Steinhardt}]{strauss2018review}
Strauss, A.~K., Klein, R., Steinhardt, C., 2018. A review of choice-based
  revenue management: Theory and methods. European Journal of Operational
  Research.

\bibitem[{Strauss and Talluri(2017)}]{strauss2017tractable}
Strauss, A.~K., Talluri, K., 2017. Tractable consideration set structures for
  assortment optimization and network revenue management. Production and
  Operations Management 26, 1359--1368.

\bibitem[{Talluri(2010)}]{talluri2010randomized}
Talluri, K., 2010. A randomized concave programming method for choice network
  revenue management. Working Papers (Departamento de Econom{\'\i}a y Empresa,
  Universitat Pompeu Fabra).

\bibitem[{Talluri(2014)}]{talluri2014new}
Talluri, K., 2014. New formulations for choice network revenue management.
  INFORMS Journal on Computing 26~(2), 401--413.

\bibitem[{Talluri and van Ryzin(2004{\natexlab{a}})}]{talluri2004revenue}
Talluri, K., van Ryzin, G., 2004{\natexlab{a}}. Revenue management under a
  general discrete choice model of consumer behavior. Management Science
  50~(1), 15--33.

\bibitem[{Talluri and van Ryzin(2004{\natexlab{b}})}]{talluri2006theory}
Talluri, K., van Ryzin, G., 2004{\natexlab{b}}. The Theory and Practice of
  Revenue Management. Vol.~68. Springer Science and Business Media.

\bibitem[{van Ryzin and Vulcano(2015)}]{van2014market}
van Ryzin, G., Vulcano, G., 2015. A market discovery algorithm to estimate a
  general class of nonparametric choice models. Management Science 61~(2),
  281--300.

\bibitem[{van Ryzin and Vulcano(2017)}]{van2017expectation}
van Ryzin, G., Vulcano, G., apr 2017. Technical note---an
  expectation-maximization method to estimate a rank-based choice model of
  demand. Operations Research 65~(2), 396--407.

\bibitem[{Zhang and Weatherford(2017)}]{zhang2016dynamic}
Zhang, D., Weatherford, L., 2017. Dynamic pricing for network revenue
  management: A new approach and application in the hotel industry. INFORMS
  Journal on Computing 29~(1), 18--35.

\end{thebibliography}

\newpage
\appendix
\label{productClosing:sec:electronicCompanion}

\section{Instances}
\label{productClosing:ec:sec:instances}

\begin{table}[h!]
\centering
\includegraphics[max width=0.99\textwidth]{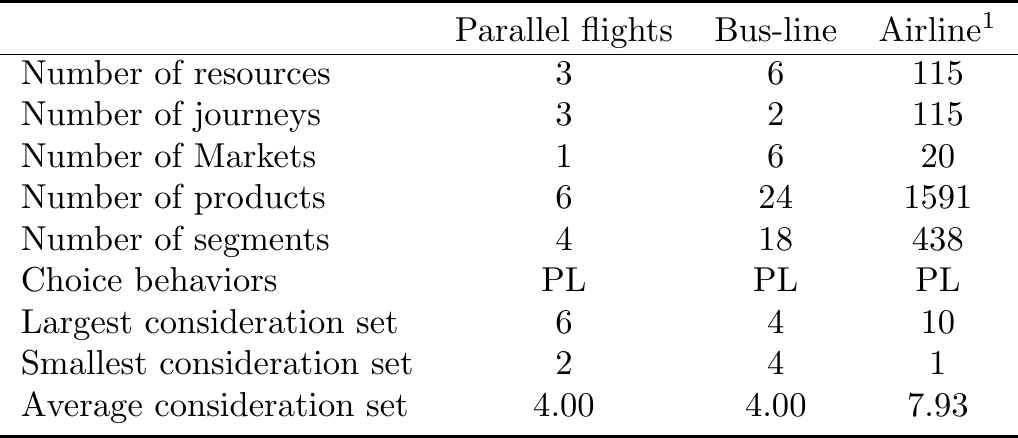}
\caption{Instances characteristics. $^{1}$ five cities. We use the preference list (PL) choice behavior as presented in the article. \label{productClosing:ec:tab:instances}}
\end{table}

Instances are entirely described in CSV files at:
\begin{figure*}[h!]
\centering
http://thibaultbarbier.com
\end{figure*}

\clearpage
\section{Parallel flights}
\label{productClosing:ec:sec:parallelFlights}

\begin{table}[h!]
\centering
\includegraphics[max width=\textwidth]{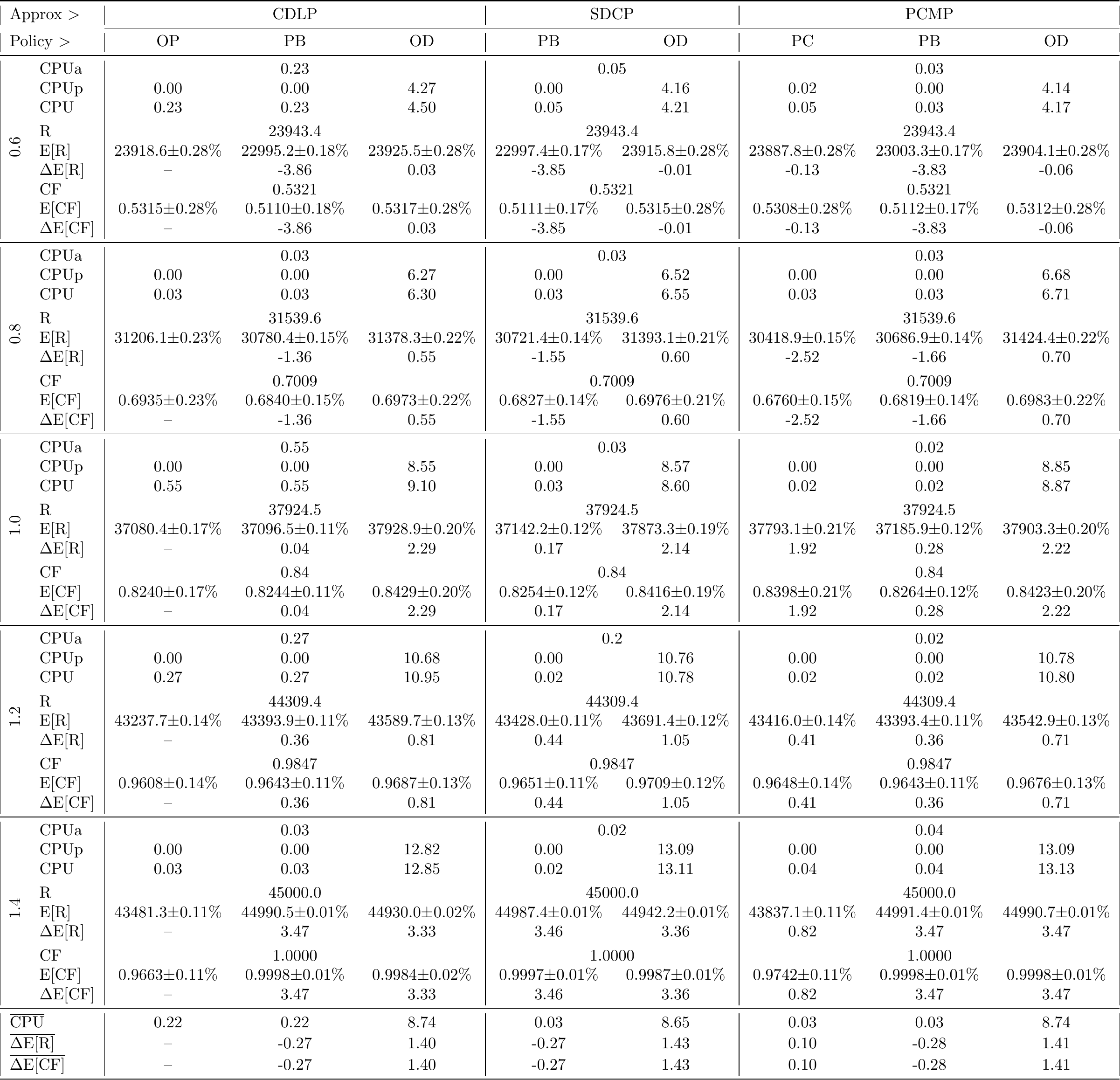}
\caption{Parallel flights results \label{productClosing:ec:tab:parallelFlights}}
{Each approximation is solved in \textbf{CPUa} seconds and return an optimal revenue \textbf{R} corresponding to a capacity factor \textbf{CF}. 
We then transform this solution to policy(ies). 
This transformation takes \textbf{CPUp} seconds and is then simulated in a discrete arrivals simulation with \textbf{3000} evaluations to obtain an expected revenue \textbf{E[R]} and expected capacity factor \textbf{E[CF]} for a 95\% confidence interval.
The total running time is \textbf{CPU} and we calculate \textbf{$\Delta$E[CF]} and \textbf{$\Delta$E[R]} the capacity factor and expected revenue relative difference with respect to CDLP-OP.}
\end{table}

\clearpage
\section{Bus-line}
\label{productClosing:ec:sec:busLine}

\begin{table}[h!]
\centering
\includegraphics[max width=\textwidth]{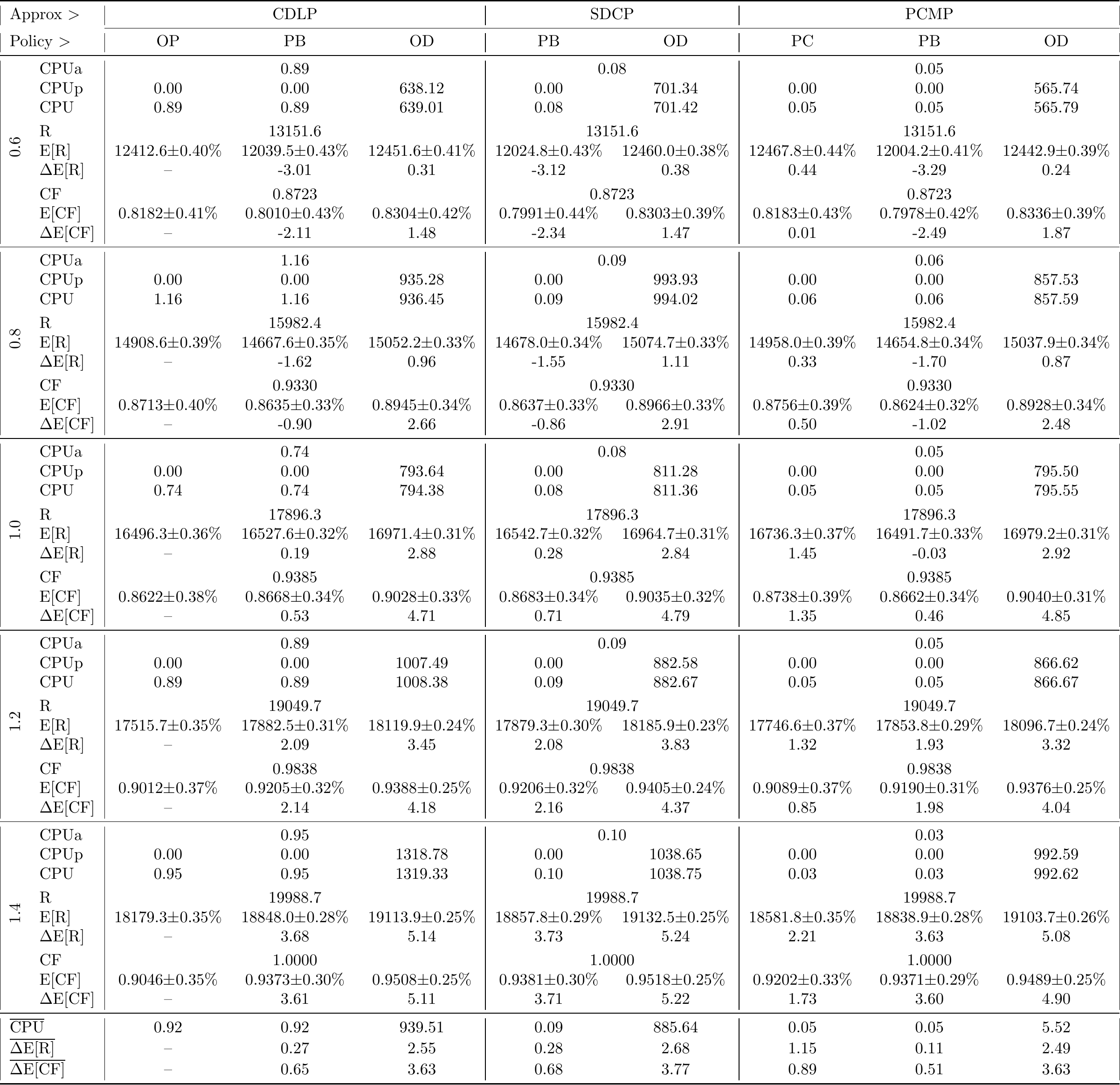}
\caption{Bus-line results \label{productClosing:ec:tab:busLine}}
{Each approximation is solved in \textbf{CPUa} seconds and return an optimal revenue \textbf{R} corresponding to a capacity factor \textbf{CF}. 
We then transform this solution to policy(ies). 
This transformation takes \textbf{CPUp} seconds and is then simulated in a discrete arrivals simulation with \textbf{1000} evaluations to obtain an expected revenue \textbf{E[R]} and expected capacity factor \textbf{E[CF]} for a 95\% confidence interval.
The total running time is \textbf{CPU} and we calculate \textbf{$\Delta$E[CF]} and \textbf{$\Delta$E[R]} the capacity factor and expected revenue relative difference with respect to CDLP-OP.}
\end{table}

\begin{table}[h!]
	\centering
\includegraphics[max width=\textwidth]{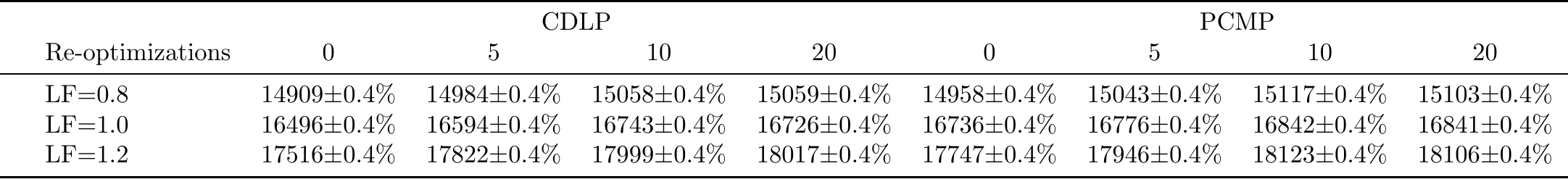}		
	\caption{Expected revenues with re-optimizations \label{productClosing:ec:tab:reopt}}
\end{table}

\clearpage
\section{Airline}
\label{productClosing:ec:sec:airline}

\begin{table}[h!]
\includegraphics[max width=\textwidth]{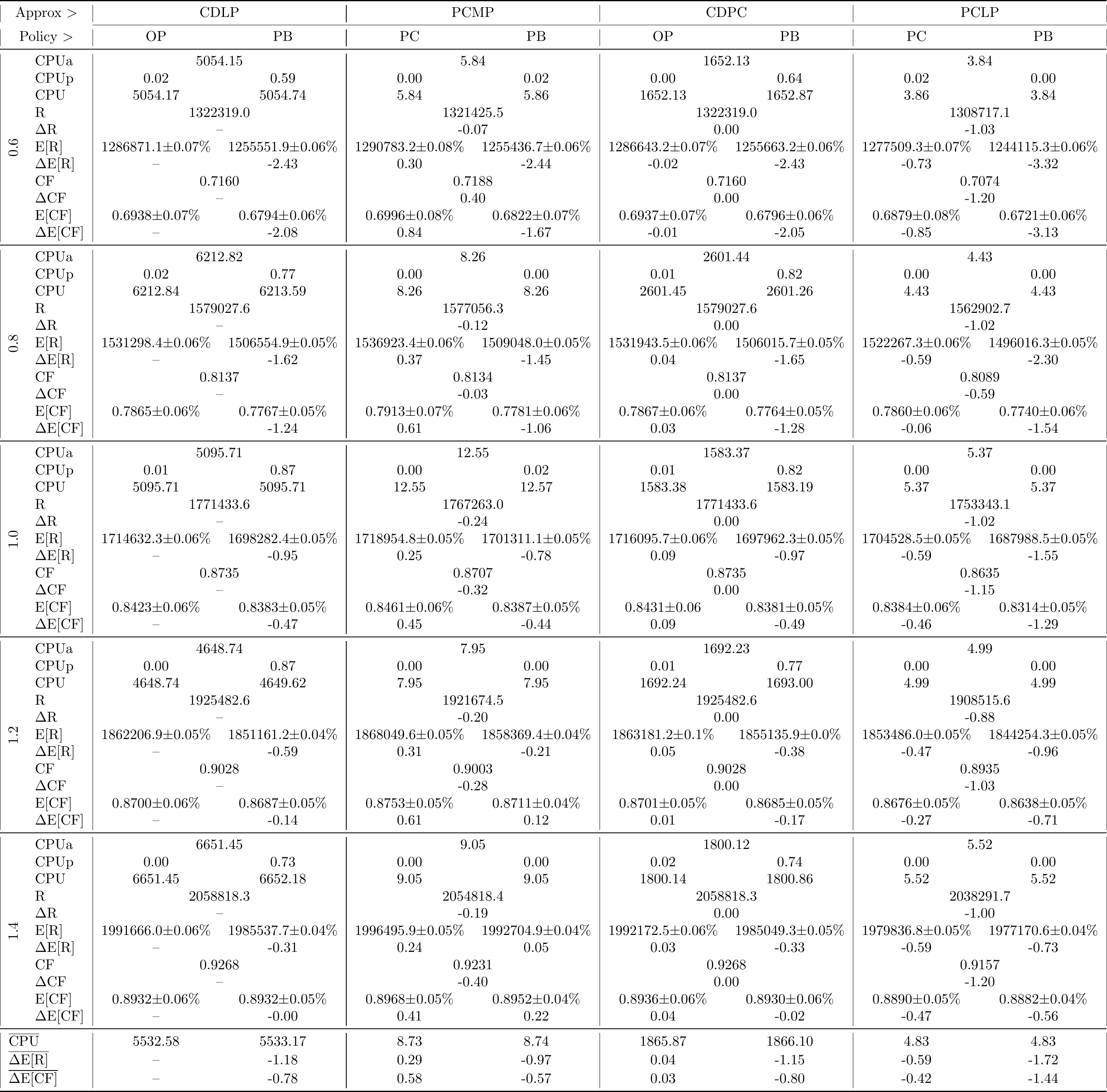}
\caption{Bus-line results \label{productClosing:ec:tab:airline}}
{Each approximation is solved in \textbf{CPUa} seconds and return a revenue \textbf{R} corresponding to a capacity factor \textbf{CF}. 
We then transform this solution to policy(ies). 
This transformation takes \textbf{CPUp} seconds and is then simulated in a discrete arrivals simulation with \textbf{500} evaluations to obtain an expected revenue \textbf{E[R]} and expected capacity factor \textbf{E[CF]} for a 95\% confidence interval.
The total running time is \textbf{CPU} and we calculate \textbf{$\Delta$E[CF]} and \textbf{$\Delta$E[R]} the capacity factor and expected revenue relative difference with respect to CDLP-OP.
}
\end{table}

\end{document}